\documentclass{article}
\usepackage[margin = 1in]{geometry}
\usepackage[utf8]{inputenc}

\usepackage{graphicx}
\usepackage{amssymb,amsmath,mathtools,amscd}
\usepackage[makeroom]{cancel}
\usepackage{amsthm}
\usepackage{mathabx}
\usepackage{epstopdf}
\usepackage{tikz}
\usepackage{tikz-3dplot}
\usetikzlibrary{arrows.meta}
\usetikzlibrary {patterns,patterns.meta}
\usetikzlibrary{shapes.geometric}
\usepackage{indentfirst}
\usepackage{hyperref}
\hypersetup{
	colorlinks=true,
	linkcolor=black,
	filecolor=black,      
	urlcolor=black,
	citecolor=blue,
}

\newcommand{\CC}{\mathbb{C}}
\newcommand{\RR}{\mathbb{R}}

\newcommand{\ZZ}{\mathbb{Z}}
\newcommand{\NN}{\mathbb{N}}
\newcommand{\kk}{\Bbbk}
\newcommand{\KK}{\mathbb{K}}

\newcommand{\bsep}{\beta_{\mathrm{sep}}}
\newcommand{\bfield}{\beta_{\mathrm{field}}}
\newcommand{\bfieldr}{\bfield^\mathbf{r}}

\newcommand{\topdeg}{\operatorname{topdeg}}
\newcommand{\Dspan}{D_\mathrm{span}}
\newcommand{\Dspanl}{D_\mathrm{span}(L)}
\newcommand{\Dreg}{D_\mathrm{reg}}
\newcommand{\Dirr}{D^\otimes_\mathrm{irr}}
\newcommand{\Char}{\operatorname{char}}

\newcommand{\bfx}{\mathbf{x}}
\newcommand{\bfa}{\mathbf{a}}
\newcommand{\bfb}{\mathbf{b}}
\newcommand{\bfc}{\mathbf{c}}
\newcommand{\bfe}{\mathbf{e}}

\newcommand{\bfu}{\mathbf{u}}
\newcommand{\bfv}{\mathbf{v}}
\newcommand{\bfw}{\mathbf{w}}
\newcommand{\bfH}{\mathbf{H}}
\newcommand{\bfE}{\mathbf{E}}
\newcommand{\bfD}{\mathbf{D}}

\newcommand{\LM}{\mathcal{LM}}
\newcommand{\Hom}{\operatorname{Hom}}

\newcommand{\Vol}{\operatorname{Vol}}

\newcommand{\wt}{\operatorname{wt}}
\newcommand{\pighere}{\qedhere}
\newcommand{\Irr}{\operatorname{Irr}}

\newtheorem{theorem}{Theorem}[section]
\newtheorem{lemma}[theorem]{Lemma}

\newtheorem{proposition}[theorem]{Proposition}

\newtheorem{observation}[theorem]{Observation}

\newtheorem*{theorem*}{Theorem}
\newtheorem*{proposition*}{Proposition}
\newtheorem*{conjecture*}{Conjecture}
\newtheorem*{corollary*}{Corollary}

\theoremstyle{definition}
\newtheorem{definition}[theorem]{Definition}
\newtheorem{example}[theorem]{Example}
\newtheorem{remark}[theorem]{Remark}
\newtheorem*{remark*}{Remark}
\newtheorem{notation}[theorem]{Notation}
\newtheorem{convention}[theorem]{Convention}

\makeatletter
\newcommand{\subjclass}[2][1991]{
  \let\@oldtitle\@title
  \gdef\@title{\@oldtitle\footnotetext{#1 \emph{Mathematics Subject Classification.} #2}}
}
\newcommand{\keywords}[1]{%
  \let\@@oldtitle\@title%
  \gdef\@title{\@@oldtitle\footnotetext{\emph{Key words and phrases.} #1.}}%
}
\makeatother

\title{Geometry of numbers and degree bounds for rational invariants}
\author{Ben Blum-Smith, Sylvan Crane, Karla Guzman, Alexis Menenses, \\ and Maxine Song-Hurewitz}
\date{\today}


\begin{document}

\maketitle

\begin{abstract}
We investigate degree bounds for fields of rational invariants of representations of finite groups. We prove many cases of a bound for $\ZZ/p\ZZ$ conjectured by Blum-Smith, Garcia, Hidalgo, and Rodriguez. For arbitrary groups, we also prove a new bound on the minimum degree $d$ such that the polynomials of degree $\leq d$ span the field of rational functions as a vector space over the invariant field. This latter quantity also bounds the degree $d$ such that the polynomials of degree $\leq d$ contain a copy of the regular representation of $G$, advancing an inquiry of Koll\'ar and Tiep. The methods involve Euclidean lattices and Minkowski's geometry of numbers.
\end{abstract}

\tableofcontents

\section{Introduction}

Invariant rings of finite groups are finitely generated as algebras, and there is a line of research stretching back over 100 years, starting with \cite{noether}, on bounding the degrees of the generators. Fields of rational invariants of finite groups are an even older object of interest, going back at least to \cite{galois1846oeuvres}; and one can recognize elements of the theory even as far back as \cite{de1770reflexions}, although it was not conceived of in that way at the time. Nonetheless, the investigation of degree bounds in this latter context seems to be comparatively young; see \cite{fleischmann2007homomorphisms, hubert-labahn, first-beam-paper, blum2024degree, edidin2025generic, edidin2025orbit, reimers2025generic, bbs-derksen}. New impetus was given to this investigation by a recent development in signal processing: when estimating a signal that has been corrupted not only by noise but by transformations randomly drawn from a group, degree bounds on generators for the field of rational invariants yield information-theoretic bounds on the number of samples one needs to view in order to accurately estimate (the orbit of) a generic signal \cite[Theorems~2.15 and 2.16 and  Corollary~3.31]{bandeira2017estimation}.

The present work is a contribution to this line of research. We make significant progress toward an upper bound for the case $G=\ZZ/p\ZZ$ ($p$ prime) that was conjectured in \cite{first-beam-paper}. We also give a bound on the degrees of generators for the ambient rational function field as a vector space over the invariant field. In this introduction, we identify the objects of study, state our main results, and discuss motivation, context, and methods.

\begin{paragraph}{Objects of study.}
Let $G$ be a finite group, $\kk$ a field, and $V$ a finite-dimensional representation of $G$ over $\kk$. As usual, let $\kk[V]^G$ and $\kk(V)^G$ denote the ring of invariants and the field of rational invariants for the action of $G$ on $V$. Our objects of study are the following three natural numbers:
\begin{equation}\label{eq:bfield}
    \bfield(G,V) := \min(d:\kk(V)^G\text{ is generated  by invariant polynomials of degree}\leq d),
\end{equation}
\begin{equation}\label{eq:bfieldr}
    \bfieldr(G,V) := \min(d:\kk(V)^G\text{ is generated by rational invariants $f_1/f_2$ with }\deg f_1 + \deg f_2\leq d),
\end{equation}
and
\begin{equation}\label{eq:Dspan}
    \Dspan(G,V) := \min(d:\kk(V)\text{ is spanned over $\kk(V)^G$ by polynomials of degree}\leq d),
\end{equation} 
When they are clear from context, we usually drop the $(G,V)$ from the notation, writing $\bfield$, $\bfieldr$, and $\Dspan$.
\end{paragraph}

\begin{paragraph}{Motivation and context.}
The quantities $\bfield$ and $\bfieldr$ are field analogues to the classical Noether number $\beta(G,V)$, the minimum $d$ such that the invariant ring $\kk[V]^G$ is generated (as a $\kk$-algebra) by the invariant polynomials of degree $\leq d$. Thus they are natural objects of study in view of the vast literature on $\beta(G,V)$---a far-from-comprehensive list of highlights includes \cite{noether, schmid1991finite, domokos-hegedus, sezer2002sharpening, fleischmann2000noether, fogarty2001noether, fleischmann2006noethermodular, symonds2011castelnuovo, cziszter-domokos, cziszter2014noether, cziszter2016interplay, gandini2019ideals, hegedHus2019finite, cziszter2019lower, ferraro2021noether, kemper2025schemes}---and the classical interest in fields of rational invariants  \cite{galois1846oeuvres, jordan1870traite, burnside1911theory, noether1913rationale}. 

They may also be viewed as analogues to  $\bsep(G,V)$, the minimum $d$ such that the polynomial invariants of degree $\leq d$ separate the orbits of $G$ on $V$, which is also the subject of a large and active  literature \cite{derksen-kemper, domokos2007typical, draisma2008polarization, kemper2009separating, sezer2009constructing, dufresne2009separating, dufresne2009cohen, kohls-kraft, domokos2011helly, dufresne2013finite, kohls2013separating,  dufresne2015separating, domokos, reimers2018separating, reimers2020separating, lopatin2021separating, domokos2022separating, kemper2022separating, domokos2024separating, schefler2025separating, 
domokos2025separating, domokosschefler2025, jia2025modular}. If $\kk$ is algebraically closed of characteristic zero, then generators for the field of rational invariants separate {\em generic} orbits \cite[Theorem~2]{rosenlicht1956some}, so $\bfield,\bfieldr$ are relaxations of $\bsep$ in this case (and $\beta$ always).

The number $\bfield$ was defined (with slightly different notation) in \cite{fleischmann2007homomorphisms}, where it was proven that $\bfield(G,V) \leq |G|$ regardless of field characteristic. It was further studied in \cite{hubert-labahn, first-beam-paper, edidin2025generic, edidin2025orbit, bbs-derksen} (and see also \cite{reimers2025generic}). In  the case $\kk=\RR$, it is the quantity of interest in the signal processing application mentioned above. Various forms of this signal processing problem (with a finite or more generally a compact $G$) are studied in  \cite{sigworth2016principles, abbe2017sample, singer2018mathematics, perry2019sample, bandeira2020optimal,   bendory2020single, fan2021maximum, bendory2022sparse, abas2022generalized,  bendory2022dihedral, bandeira2017estimation, edidin2024orbit, edidin2025orbit, bendory2025generalized}; see  \cite[Sections~2 and 3]{bandeira2017estimation} for the connection to $\bfield$. A much more detailed discussion of motivation and context for $\bfield$ is found in the introduction to \cite{first-beam-paper}.

The alternative $\bfieldr$ was studied in \cite{blum2024degree},  motivated by the intuition that because $\kk(V)^G$ is a field, it makes sense to consider rational but not necessarily polynomial generators. Trivially, we have $\bfieldr \leq \bfield$, so $\bfieldr$ can be viewed as a relaxation of $\bfield$. There is  equality under certain conditions, e.g., \cite[Proposition~3.1]{blum2024degree}, and, at least for abelian $G$ in coprime characteristic, the same lower bounds hold (and are sharp) for both; compare \cite[Theorem~3.1]{first-beam-paper} with \cite[Theorem~4.1]{blum2024degree}. The state of current knowledge is more preliminary with respect to upper bounds, but at least in the case $G=\ZZ/p\ZZ$, the same upper bounds are conjectured; see below.

The quantity $\Dspan$ is a field analogue to another quantity of longstanding interest in invariant theory: the minimum degree required for not-necessarily-invariant polynomials  to generate $\kk[V]$ as a module over $\kk[V]^G$, denoted $\operatorname{topdeg}(G,V)$ or $\beta(\kk[V],\kk[V]^G)$, which was studied in \cite{schmid1991finite, kohls2014top, cziszter2019lower} (and see also \cite{cziszter2013generalized, cziszter-domokos, cziszter2014noether}, which work with a common generalization of $\beta(G,V)$ and $\operatorname{topdeg}(G,V)$). One has immediately that
\[
\Dspan(G,V) \leq \topdeg(G,V),
\]
since module generators for $\kk[V]$ over $\kk[V]^G$ also span $\kk(V)$ over $\kk(V)^G$. The investigation of $\Dspan$ was suggested by Victor Reiner \cite[Question~5.11]{first-beam-paper}, motivated by the analogy with $\topdeg$. We have two additional motivations for its study:

Two quantities of interest in representation theory are:
\begin{itemize}
    \item the minimum $d$ such that the tensor powers of $V$ up to the $d$th have every irreducible representation of $G$ as a constituent (called $\Dirr$ in \cite{bbs-derksen}; see also \cite{brauer1964note, wehlau2003some, steinberg2014burnside}), and 
    \item the minimum $d$ such that $\kk[V]_{\leq d}$ contains a copy of the regular representation of $G$ (called $\Dreg$ in \cite{bbs-derksen}; see also \cite{kollar2024simple}).
\end{itemize} 
The quantity $\Dreg$ was studied recently by Koll\'ar and Tiep \cite{kollar2024simple}, who proved that $\Dreg \leq |G|-1$. In general, $\Dspan$ upper-bounds both $\Dirr$ and $\Dreg$, and in the situation of Sections~\ref{sec:preliminaries} and \ref{sec:bfieldr} of the present paper, i.e., abelian $G$ and coprime-characteristic $\kk$, all three numbers are {\em equal} \cite[Proposition~4.1]{bbs-derksen}. Thus, by studying $\Dspan$ we also study these quantities. A more detailed discussion of the motivation for $\Dspan$ coming from representation and invariant theory is found in the introduction to \cite{bbs-derksen}.

A final motivation for studying $\Dspan$ is that it directly bounds $\bfield$. It was shown very recently \cite[Theorem~1.1]{bbs-derksen} that $\Dspan$ and $\bfield$ satisfy the inequality
\[
\bfield \leq 2\Dspan + 1;
\]
thus, advancing our understanding of $\Dspan$ may ultimately help get a hold of $\bfield$. (That said, this hope is not yet realized in the present work: the result on $\Dspan$ proven below yields no new information about $\bfield$.)

\end{paragraph}

\begin{paragraph}{Prior state of the art; main results.}
As throughout, $G$ is a finite group and $V$ a representation of $G$ over a field $\kk$, of finite dimension $N$. The following is a survey of results on $\bfield$, $\bfieldr$, and $\Dspan$ that appear in the existing literature:
\begin{itemize}
    \item $\bfield \leq |G|$ with no additional hypotheses \cite[Corollary~2.3]{fleischmann2007homomorphisms}. (In particular, this inequality holds even if $\Char \kk$ divides $|G|$, providing a contrast with the classical Noether number $\beta(G,V)$ which is always $\leq |G|$ when $\Char \kk\nmid |G|$ \cite{noether, fleischmann2000noether, fogarty2001noether}, but not otherwise \cite{richman1996invariants}.)
    \item $\Dspan \leq |G|-1$ with no aditional hypotheses \cite{bbs-derksen}. (This is parallel to the previous bullet. The ``classical" analogue to $\Dspan$ is $\operatorname{topdeg}(G,V)$, defined above, which satisfies $\operatorname{topdeg}(G,V)\leq |G|-1$ when $\Char \kk \nmid |G|$, but not otherwise \cite{kohls2014top}.) 
    
    \item $\bfield \geq \sqrt[N]{|G|}$ with no additional hypotheses \cite[Theorem~3.2]{first-beam-paper}.

    \item If $G$ is abelian and $\Char\kk$ does not divide $|G|$, let $m$ be the number of nontrivial isotypic components in $V$'s base change to the algebraic closure $\overline \kk$. (Note that $m\leq N$.) Then $\bfield \geq \sqrt[m]{|G|}$ \cite[Theorem~3.1]{first-beam-paper}, and the same bound even holds for $\bfieldr$ \cite[Theorem~4.1]{blum2024degree}.

    \item Under the same hypotheses as the previous bullet, $\bfield \geq 3$ unless $G$ is an elementary abelian 2-group \cite[Proposition~3.5]{first-beam-paper}; the same bound holds for $\bfieldr$ \cite[Proposition~4.3]{blum2024degree}. If $V$ is the regular representation of $G$, the same assertion for $\bfield$ holds in characteristic zero without the abelian assumption \cite[Section~4.3.3]{bandeira2017estimation}. (The same inequality likely holds more generally.)
    
    \item If the characteristic of $\kk$ is zero and $V$ contains the regular representation of $G$, then $\bfield \leq 3$ \cite{edidin2025orbit}.
    
    \item As mentioned above, if the characteristic of $\kk$ is coprime with $|G|$, then $\bfield \leq 2\Dspan + 1$ \cite{bbs-derksen}. (This generalizes the previous bullet, whose hypothesis on $V$ implies that $\Dspan = 1$.)
    
    \item Suppose $G=\ZZ/p\ZZ$, $p$ prime, and $\Char \kk$ is different from $p$. With $m$ as above, suppose $m\geq 3$. Then $\bfield \leq (p+3)/2$ \cite[Theorem~3.11]{first-beam-paper}.
\end{itemize}

For non-modular representations of $G=\ZZ/p\ZZ$ with $m=2$ that satisfy a certain easily-checked condition, one also has an explicit formula for $\bfield$ \cite[Proposition~4.2]{first-beam-paper}. Furthermore, for certain specific representations of various important groups, $\bfield$, $\bfieldr$, and/or $\Dspan$ are also known, or easily deduced, from known explicit information about the invariant ring and the representation theory. 

The upper bound in the last bullet above is not sharp. Based on computational data, the following is expected. If correct, it is sharp even for $\bfieldr$ \cite[Proposition~4.5]{blum2024degree}.

\begin{conjecture*}[Conjecture~5.1 in \cite{first-beam-paper}]\label{conj:the-conjecture}
    Let $G=\ZZ/p\ZZ$ with $p$ an odd prime, $\kk$ a field of characteristic different from $p$, and $m$ the number of distinct, nontrivial isotypic components in $V$'s base change to an algebraic closure $\overline\kk$. Then
    \[
    \bfield \leq \left\lceil\frac{p}{\lceil m/2\rceil}\right\rceil.
    \]
\end{conjecture*}

A special case is known: if $V$ is the regular representation of $\ZZ/p\ZZ$, then the bound $\bfield \leq 3$ proven in \cite{edidin2025orbit} matches the conjectured bound. Our first main result is that, for fixed $m$, the conjectured bound holds for all but finitely many $p$ provided that the set of characters appearing in $V$ is closed under inversion:

\begin{theorem*}[Theorem~\ref{thm:real-reps-bfield} below]
    Let $G=\ZZ/p\ZZ$ with $p$ an odd prime, $\kk$ a field of characteristic different from $p$, and $m$ the number of distinct, nontrivial isotypic components in $V$'s base change to an algebraic closure $\overline \kk$. Suppose that for any nontrivial character appearing in $V_{\overline \kk}$, the inverse character also appears (note that this implies  $m$ is even). Then, for fixed $m$ and all sufficiently high $p$, we have
    \[
    \bfield \leq \lceil 2p/m\rceil.
    \]
\end{theorem*}

The hypothesis on $V$ in this theorem is automatically satisfied if $\kk=\RR$, so in particular, the conjectured bound holds for all but finitely many $p$ given fixed $m$ in the situation of relevance to the motivating signal processing application. The theorem is deduced as a corollary of our next main result, which allows to drop the extra hypothesis on $V$ at the price of relaxing $\bfield$ to $\bfieldr$:

\begin{theorem*}[Theorem~\ref{thm:gen-deg} below, first statement]
    Let $G=\ZZ/p\ZZ$, $p$ an odd prime. Let $\kk$ be a field containing distinct $p$th roots of unity. Let $V$ be a finite-dimensional representation of $G$ over $\kk$, with $m$ distinct, nontrivial isotypic components.

    For fixed $m$ and all sufficiently high $p$, we have
    \[
    \bfieldr \leq \left\lceil\frac{p}{\lceil m/2\rceil}\right\rceil.
    \]
\end{theorem*}

The sharpness of the bound in Conjecture~\ref{conj:the-conjecture} (if it holds) is known by an explicit construction of representations that attain the bound (see Proposition~\ref{prop:sharp-case} below). It was asked in \cite[Question~5.4]{first-beam-paper} whether, for fixed $m$ and sufficiently high $p$, these are the only ones that attain it. The second statement in Theorem~\ref{thm:gen-deg} below shows that this holds if we replace $\bfield$ with $\bfieldr$.

Our final main result improves on the above general bound $\Dspan \leq |G|-1$ in the situation that the field characteristic is coprime, as long as $V$ is not a very special representation of a cyclic group:

\begin{theorem*}[Theorem~\ref{thm:Dspan-main} below]
    Let $G$ be a finite group, and $\kk$ a field of characteristic coprime with $|G|$. Let $V$ be a faithful, finite-dimensional representation of $G$ over $\kk$ that is not scalar or a sum of a scalar and a trivial representation. Then
    \[
    \Dspan \leq |G|/2.
    \]
\end{theorem*}

While this theorem is stated for any finite group, its main interest is if $G$ is cyclic or dicyclic.  As mentioned above,  $\Dspan(G,V)$ is a field analogue to $\operatorname{topdeg}(G,V)$, the minimum degree of polynomials needed to generate $\kk[V]$ as a $\kk[V]^G$-module. As mentioned above, we immediately have $\Dspan(G,V) \leq \operatorname{topdeg}(G,V)$. By combining  \cite[Theorem~1]{kohls2014top}, with results of \cite{cziszter-domokos, cziszter2014noether}, one obtains  $\operatorname{topdeg}(G,V) \leq |G|/2$ when $G$ is not cyclic or dicyclic (and $\Char\kk\nmid |G|$); therefore $\Dspan(G,V) \leq |G|/2$ as well---this inference is carried out carefully in the proof of Theorem~\ref{thm:Dspan-main}. In the dicyclic case, the same reasoning leads to $\Dspan(G,V)\leq |G|/2+1$.  So this theorem on $\Dspan$ is an improvement on a known bound on $\operatorname{topdeg}$ only in the cyclic and dicyclic cases, and in the dicyclic case, only by $1$.\footnote{It is likely that the bound in Theorem~\ref{thm:Dspan-main} is not sharp for dicyclic groups; however, we did not make an effort to push the result for dicyclic groups further than the general statement in the theorem.} On the other hand, if $G$ is cyclic and $V$ contains any faithful one-dimensional representation of $G$, then $\operatorname{topdeg}(G,V)$ is $|G|-1$ exactly. So the theorem tells us that $\Dspan$ is at most about half of $\operatorname{topdeg}$ whenever this holds but $V$ is not essentially scalar. Thus, the cyclic case is the main case.

\end{paragraph}

\begin{paragraph}{Methods and structure.}
The starting point for the proofs of all of our main results is a common strategy in the invariant theory of finite abelian groups. Assuming $\kk$ contains sufficiently many roots of unity, the action of $G$ on $V$ can be diagonalized. Choosing a diagonal basis for $V^*$, the action of $G$ on $\kk[V] \cong \operatorname{Sym}(V^*)$ is then a monomial action, and the invariant ring $\kk[V]^G$ is the  semigroup ring over the normal affine semigroup of invariant monomials. Then properties of the invariant ring can be deduced from properties of the underlying semigroup; see for example \cite{huffman, schmid1991finite, bruns-herzog, smith1996noether, neusel-sezer, cziszter2016interplay, domokos, gandini2019ideals}. Similarly, the field $\kk(V)^G$ of rational invariants can be studied by analyzing the group of invariant Laurent monomials \cite{hubert-labahn, first-beam-paper, blum2024degree}, which can be viewed as a Euclidean lattice.

In particular, it is shown in \cite{first-beam-paper}, respectively \cite{blum2024degree}, that $\bfield$, respectively $\bfieldr$, can be determined in this setting directly from the lattice of invariant Laurent monomials. In particular, $\bfieldr$ is precisely the smallest natural number $d$ so that the group of invariant Laurent monomials, viewed as a Euclidean lattice, is generated by its members of $L^1$-norm $\leq d$.  In a similar way, it will be deduced below from results in \cite{first-beam-paper} and \cite{bbs-derksen} that $\Dspan$ can be determined from the way the lattice of invariant Laurent monomials sits inside the lattice of all Laurent monomials. Almost all of the results in the present work are proven by analyzing these lattices. We quote the needed material from \cite{first-beam-paper, blum2024degree, bbs-derksen} below in Section~\ref{sec:preliminaries}.

A key ingredient in the proof of Theorem~\ref{thm:gen-deg}, bounding $\bfieldr$ in the case $G=\ZZ/p\ZZ$, is {\em Minkowski's geometry of numbers} \cite{minkowski1910geometrie}. This body of theory concerns interactions between lattices and convex sets in Euclidean space; a modern reference is \cite{lekkerkerker-gruber}. It provides us with theorems that, given a lattice and a vector space norm, will allow us to find a basis for that lattice such that all but one of the basis vectors are short with respect to that norm. We work with these short vectors, and the equation that defines the lattice of invariant Laurent monomials, to find an additional vector that completes the basis and is not too long. Theorem~\ref{thm:real-reps-bfield} is deduced from Theorem~\ref{thm:gen-deg}.

Meanwhile, Theorem~\ref{thm:Dspan-main}, which is our main result on $\Dspan$, is proven with a case analysis depending on whether  $G$ is cyclic,  dicyclic, or neither. For groups that are neither cyclic nor dicyclic, we deduce it with essentially no work from known results on $\beta(G,V)$ and $\operatorname{topdeg}(G,V)$. For cyclic groups, we prove it along lines reminiscent, in broad outline, of the proof of \cite[Theorem~3.11]{first-beam-paper}, which asserts that $\bfield(\ZZ/p\ZZ,V)\leq (p+3)/2$ if $V$ contains at least three nontrivial isotypic components. Although the details differ greatly, both proofs involve bootstrapping from detailed information on the $m=2$ case that is obtained using convexity arguments. For dicyclic groups, we use a characterization of $\Dspan$ in terms of representation theory, developed in \cite{bbs-derksen}, to compute $\Dspan$ explicitly for a faithful irreducible representation, and then bootstrap to an arbitrary representation. For both the cyclic and dicyclic cases, the bootstrapping technique involves estimates obtained from a combinatorial analysis of the (order-theoretic) lattice of normal subgroups of $G$.

The paper is organized as follows. Section~\ref{sec:preliminaries} summarizes the reduction of the study of $\bfield$, $\bfieldr$, and $\Dspan$ to lattices, and collects various  lemmas from the literature. Section~\ref{sec:bfieldr} proves Theorem~\ref{thm:gen-deg} and deduces Theorem~\ref{thm:real-reps-bfield} from it. Section~\ref{sec:HRD} proves Theorem~\ref{thm:Dspan-main}, in two steps: Section~\ref{sec:pf-of-HRD-main} proves Theorem~\ref{thm:HRD-main}, a result on lattices in the $m=2$ case that forms the basis of the cyclic case, and Section~\ref{sec:pf-of-Dspan-main} uses this to complete the cyclic case, and handles the rest of the cases.
\end{paragraph}

\section{Preliminaries}\label{sec:preliminaries}

As mentioned above, the starting point for  our proofs, except for the nonabelian part of Theorem~\ref{thm:Dspan-main}, is to reduce the invariant theory to a problem about sublattices of the integer lattice, and to study the latter. In this section, we walk briefly through this reduction, and gather together the tools required to study the resulting lattices, including a classical theorem from Minkowski's geometry of numbers. We also describe a characterization of $\Dspan$ in terms of representation theory, developed in \cite{bbs-derksen}, that will be used in the nonabelian part of Theorem~\ref{thm:Dspan-main}.

\begin{notation}
    Throughout, $G$ is a finite group; $\kk$ is a field of characteristic coprime to the order of $G$; and $V$ is a finite-dimensional, faithful representation of $G$ over $\kk$, of dimension $N$. The group algebra of $G$ over $\kk$ is $\kk G$. 
    
    In this section, up until Notation~\ref{not:irr}, $G$ is stipulated to be abelian. The character group of $G$, i.e., the group (under pointwise multiplication) of group homomorphisms from $G$ to the multiplicative group of roots of unity in an algebraic closure of $\kk$, is denoted $\widehat G$. 
\end{notation}

If $\kk$ contains enough roots of unity, then the action of $G$ on $V$ is diagonalizable. Choosing a basis $x_1,\dots,x_N$ that diagonalizes the induced action on $V^*$, all Laurent monomials in the $x_j$'s are then eigenvectors for the induced action on $\kk(V)$. With respect to the results on $\bfield$ and $\Dspan$, this hypothesis costs no generality, because neither $\bfield$ (by \cite[Lemma~2.1]{first-beam-paper}) nor $\Dspan$ (by \cite[Lemma~2.2]{bbs-derksen}) is affected by base change. However, our main result on $\bfieldr$ does require this hypothesis on $\kk$, as it is an upper bound, and $\bfieldr$ can be lowered by base change \cite[Example~3.3]{blum2024degree}. 

\begin{convention}\label{con:hypothesis-on-k}
    For the rest of this section, we assume that $\kk$ contains $|G|$th roots of unity, and that rational functions are written with respect to a basis $x_1,\dots,x_N$ for $V^*$ that diagonalizes the action of $G$. But in theorem statements in later sections, we state explicit hypotheses on $\kk$ as needed.
\end{convention}

All of the following notation follows \cite{first-beam-paper}.

\begin{notation}\label{not:LM}
    Let $\LM$ denote the group of Laurent monomials in $x_1,\dots,x_N$ under multiplication, and $\LM^G$ the subgroup of invariant Laurent monomials. 
\end{notation}

\begin{notation}\label{not:log-exp}
    Sending a monomial to its exponent vector, we obtain an identification of $\LM$ with the integer lattice $\ZZ^N$. Let $\log:\LM \rightarrow \ZZ^N$ denote this identification, and $\exp:\ZZ^N\rightarrow \LM$ its inverse.
\end{notation}

\begin{convention}
    We adopt the convention that $0\in \NN$, so that the ordinary monomials in $\LM$ have exponent vectors in $\NN^N$, the nonnegative orthant of $\ZZ^N$.
\end{convention}

\begin{convention}
    Throughout this work, the word {\em lattice} refers to Euclidean lattices, and specifically $\ZZ^N$ and its subgroups. (We do not insist that they have full rank.) We will have an occasion in Section~\ref{sec:pf-of-Dspan-main} to mention lattices in the sense of order theory (i.e., posets with meet and join operations) as well; when this comes up, we will always modify the word {\em lattice} with the phrase {\em order-theoretic} to avoid ambiguity.
\end{convention}

The key definition is this:

\begin{definition}\label{def:L(G,V)}
    Denote by $L(G,V)$  the sublattice $\log(\LM^G)$ of $\ZZ^N$ consisting of exponent vectors of invariant Laurent monomials.
\end{definition}

This definition is important because $\bfield$, $\bfieldr$, and $\Dspan$ (defined in \eqref{eq:bfield}, \eqref{eq:bfieldr}, and \eqref{eq:Dspan}) can all be computed from $L(G,V)$ (Lemma~\ref{lem:lattice-knows-all} below). To prepare, we interpret the cosets of $L(G,V)$ in $\ZZ^N$ in terms of the action of $G$ on $\kk(V)$.

\begin{notation}
    Let $\Theta:\LM\rightarrow \widehat G$ be the map that sends a Laurent monomial to the character by which $G$ acts on it.
\end{notation}

By the following lemma, we know the index of $L(G,V)$ in $\ZZ^N$, which is also its covolume in $\RR^N$.

\begin{lemma}[Lemma~2.2 in \cite{first-beam-paper}]\label{lem:index-is-|G|}
The group homomorphism $\Theta\circ\exp: \ZZ^N\rightarrow\widehat{G}$ is surjective with kernel $L(G,V)$. Thus, the cosets of $L(G,V)$ in $\ZZ^N$ are in bijection with the characters of $G$ along this homomorphism, and the index of $L(G,V)$ in $\ZZ^N$ is the order of $G$.
\end{lemma}

Now we can explain how to determine $\bfield$, $\bfieldr$, and $\Dspan$ from $L(G,V)$.

\begin{lemma}\label{lem:lattice-knows-all}
    The following hold:
    \begin{enumerate}
        \item $\bfield(G,V)$ is the smallest $d\in \NN$ such that $\NN^N$ contains a generating set for $L(G,V)$ of $L^1$ norm $\leq d$.\label{item:bfield-from-lattice}
        \item $\bfieldr(G,V)$ is the smallest $d\in \NN$ such that $\ZZ^N$ contains a generating set for $L(G,V)$ of $L^1$ norm $\leq d$.\label{item:bfieldr-from-lattice}
        \item $\Dspan(G,V)$ is the smallest $d\in \NN$ such that $\NN^N$ contains at least one representative of $L^1$ norm $\leq d$ from each coset of $L(G,V)$ in $\ZZ^N$.\label{item:Dspan-from-lattice}
    \end{enumerate}
\end{lemma}

\begin{proof}
    Statement~\ref{item:bfield-from-lattice} is part of  \cite[Lemma~2.6]{first-beam-paper}. Statement~\ref{item:bfieldr-from-lattice} is part of \cite[Lemma~3.8]{blum2024degree}. For statement~\ref{item:Dspan-from-lattice}, first, it follows from \cite[Lemma~2.5]{first-beam-paper}, with $L=\LM$ and $L'=\LM^G = \exp(L(G,V))$, that any complete set of coset representatives for $L(G,V)$ in $\ZZ^N$ yields (via $\exp$) a $\kk(V)^G$-basis of $\kk(V)$. Therefore, $\Dspan\leq d$ for $d$ as in statement~\ref{item:Dspan-from-lattice}. Conversely, from \cite[Lemma~4.1]{bbs-derksen}, $\Dspan\geq \Dreg$, where $\Dreg$ is the minimum $d$ such that $\kk[V]_{\leq d}$ contains the regular representation as a $\kk G$-module summand. By Lemma~\ref{lem:index-is-|G|}, the characters of $G$ are in bijection with the cosets of $L(G,V)$ in $\ZZ^N$. Meanwhile, $\kk[V]_{\leq d}$ is the $\kk$-span of the part of $\LM$ that is the $\exp$-image of the part of $\NN^N$ of $L^1$-norm $\leq d$. Thus $\kk[V]_{\leq d}$ contains the regular representation precisely when the part of $\NN^N$ of $L^1$-norm $\leq d$ contains a representative from each coset of $L(G,V)$ in $\ZZ^N$. It follows that $\Dreg$ is precisely $d$ as in statement~\ref{item:Dspan-from-lattice}. So the $d$ of statement~\ref{item:Dspan-from-lattice} satisfies $d=\Dreg\leq \Dspan \leq d$, and we have the promised equality.
\end{proof}

\begin{convention}\label{conv:b(L)-notation}
    Lemma~\ref{lem:lattice-knows-all} states that the quantities $\bfield$, $\bfieldr$, and $\Dspan$ are actually functions of the lattice $L(G,V)$. In view of this, if $L\subseteq \ZZ^N$ is any full-rank sublattice of the integer lattice, we can write $\bfield(L)$,  $\bfieldr(L)$, and $\Dspan(L)$ for the quantities in parts~\ref{item:bfield-from-lattice}, \ref{item:bfieldr-from-lattice}, and \ref{item:Dspan-from-lattice} respectively of Lemma~\ref{lem:lattice-knows-all}. In this language, the lemma states that $\bfield(G,V) = \bfield(L(G,V))$, and similarly for $\bfieldr, \Dspan$. As with the original definitions in terms of $G,V$, we can drop the $L$ when it is clear from context.
\end{convention}

Using $L(G,V)$, we can show that the quantities of interest are insensitive to multiplicity, or the presence of trivial characters, in $V$:

\begin{lemma}\label{lem:trivial-and-duplicate}
    The quantities $\bfield(G,V)$, $\bfieldr(G,V)$, and $\Dspan(G,V)$ depend only on the set of distinct, nontrivial characters occurring in $V$, and not on their multiplicities.
\end{lemma}

\begin{proof}
    The statement about $\bfield$, respectively $\bfieldr$, is \cite[Lemma~2.12]{first-beam-paper}, respectively \cite[Lemma~3.11]{blum2024degree}. We see the statement about $\Dspan$ using a similar technique, as follows.
    
    We show that if a subrepresentation $V'$ is obtained from $V$ by deleting a trivial or duplicate character, then $\Dspan(G,V)\geq \Dspan(G,V')$. Neither type of change can destroy faithfulness of $V$, and in either case $V$ surjects onto $V'$ as a $G$-representation, so the inequality $\Dspan(G,V') \geq \Dspan(G,V)$ also holds, by \cite[Theorem~4.4]{bbs-derksen}, and we have equality. The statement about $\Dspan$ in the lemma then follows immediately by induction. For convenience, we order the coordinate functions on $V$ so that the one to be deleted is the last one.
    
    If $V=V'\oplus \chi_1$, where $\chi_1$ is the trivial character, then 
    \[
    L(G,V) = L(G,V')\oplus \ZZ\subseteq \ZZ^{N'}\oplus\ZZ = \ZZ^N,
    \]
    and the coset of $L(G,V)$ to which a point of $\ZZ^N$ belongs does not depend on its final coordinate; thus, the representative of a coset in $\NN^N$ of minimum $L^1$ norm already occurs in $\ZZ^{N'}\oplus \{0\}$, so it already represents the corresponding coset of $L(G,V')$ in $\ZZ^{N'}$, and $\Dspan(G,V')\leq \Dspan(G,V)$.

    If $V=W\oplus \chi \oplus \chi$ for some character $\chi$, and $V'=W\oplus \chi$, then the $\Theta\circ \exp$-image of a point $(\dots,a,b)\in \NN^N$ is the same as the $\Theta\circ\exp$-image of the corresponding point $(\dots,a+b)\in \NN^{N'}$ that is equal in the first $N-2=N'-1$ coordinates. They also have the same $L^1$ norm. Since the $\Theta\circ\exp$-images correspond bijectively with cosets by Lemma~\ref{lem:index-is-|G|}, it follows that for any natural number $d$, the set of points of $\NN^N$ of $L^1$ norm $\leq d$ meets every coset of $L(G,V)$ only if the set of points of $\NN^{N'}$ of $L^1$-norm $\leq d$ meets every coset of $L(G,V')$. This completes the proof.
\end{proof}

Lemma~\ref{lem:trivial-and-duplicate} explains the presence of the quantity $m$, the number of nontrivial isotypic components of $V$, in the main results of this paper on $\ZZ/p\ZZ$. It also allows us, without loss of generality, to restrict our attention to representations $V$ that do not contain any trivial or any duplicate characters (so that $N=m$). 

\begin{convention}
    We assume going forward that $V$ contains no trivial or duplicate characters.
\end{convention}

While the representation $V$ is not in general determined by the lattice $L(G,V)$---for example, twisting by an automorphism of $G$ can change $V$ but will not change $L(G,V)$---the property of not containing trivial or duplicate characters can be read from $L(G,V)$, as follows.

\begin{observation}\label{obs:points-from-triv-and-identical}
    The representation $V$ contains a trivial character if and only if the lattice $L(G,V)$ contains one of the standard basis vectors $\bfe_i$. It contains a pair of identical characters if and only if $L(G,V)$ contains a point $\bfe_i-\bfe_j$ ($i\neq j$).
\end{observation}

To work with $L(G,V)$, it is convenient to be able to express it explicitly in terms of equations. If $G=\langle g\rangle$ is cyclic of order $n$, with generator $g$, then the characters of $G$ can be represented by integers $A\in\ZZ$ via the correspondence
\begin{equation}\label{eq:integers-characters}
A\leftrightarrow (g\mapsto \zeta^A),
\end{equation}
where $\zeta$ is some fixed primitive $n$th root of unity in $\kk$; note that $A$ is only determined mod $n$. Then, if $A_1,\dots,A_m\in \ZZ$ represent the characters $\Theta(x_1),\dots,\Theta(x_m)$ by which $G$ acts on $x_1,\dots,x_m$, the lattice $L(G,V)$ is the set of integer solutions $(a_1,\dots,a_m)\in \ZZ^m$ to the equation
\begin{equation}\label{eq:lattice-equation}
    A_1a_1 + \dots + A_ma_m = 0\pmod n.
\end{equation}
We use this frequently below. Similarly, if $G$ is (abelian but) not cyclic, say with generators $g_1,\dots,g_r$ of orders $n_1,\dots,n_r$ respectively, then a character of $G$ can be specified by an $r$-tuple of integers $(A_1,\dots,A_r)$ via the correspondence
\[
(A_1,\dots,A_r) \leftrightarrow (g_i \mapsto \zeta_i^{A_i}, \; i=1,\dots,r),
\]
where each $\zeta_i$ is a fixed primitive $n_i$th  root of unity, and each $A_i$ is determined mod $n_i$. Then $L(G,V)$ is described by a system of $r$ equations
\begin{align*}
    A_{11}a_1 + \dots + A_{1m}&a_m = 0\pmod {n_1},\\
    &\vdots\\
    A_{r1}a_1 + \dots + A_{rm}&a_m = 0\pmod {n_r},
\end{align*}
where for $j=1,\dots,m$, the tuple $(A_{1j},\dots,A_{rj})$ represents the character by which $G$ acts on $x_j$.

For use in Section~\ref{sec:HRD}, we restate Observation~\ref{obs:points-from-triv-and-identical} in terms of these equations in the special case that $m=2$. In this case, if it happens that $V$ contains a trivial character or a pair of identical characters, then it factors through a one-dimensional representation, so $G$ is necessarily cyclic and $L(G,V)$ is necessarily described by a single equation mod $n$.

\begin{observation}\label{obs:points-from-triv-and-identical-2}
    If $m=2$, then $L(G,V)$ contains the point $(1,0) = \bfe_1$, respectively $(0,1) = \bfe_2$, respectively $(1,-1) = \bfe_1-\bfe_2$, if and only if it is defined by an equation $A_1a_1+A_2a_2=0\pmod n$ with $A_1$, respectively $A_2$, respectively $A_1-A_2$, equal to $0$ mod $n$.
\end{observation}

We make use of a classical result from geometry of numbers, {\em Minkowski's Second Theorem} (also known as {\em Minkowski's Theorem on Successive Minima}). If $\Delta$ is a convex, centrally symmetric subset of $\RR^N$ containing a neighborhood of the origin, and $L$ is a lattice, then the {\em successive minima} $\lambda_1,\dots,\lambda_N$ of $\Delta$ with respect to $L$ are defined to be
\[
\lambda_i = \lambda_i(\Delta,L) := \inf(r>0 : \dim \langle L\cap r\Delta\rangle_{\RR}\geq i),
\]
where $r\Delta=\{rx\in \RR^N : x\in \Delta\}$ is the Minkowski $r$-dilation of $\Delta$, and $\langle - \rangle_{\RR}$ denotes the $\RR$-span. The following upper bound on the product of the successive minima was proven by Minkowski (\cite[F\"unftes Kapitel, \S~50]{minkowski1910geometrie}, see also \cite[Chapter~2, Section~ii, equation (11)]{lekkerkerker-gruber}):

\begin{theorem*}[Minkowski's Second Theorem]
    With notation as above, we have
    \begin{equation}\label{eq:minkowski-2nd}
    \lambda_1\dots\lambda_N \Vol(\Delta) \leq 2^N\det(L).
    \end{equation}
\end{theorem*}

In \cite{first-beam-paper}, a class of representations $V$ of $G=\ZZ/p\ZZ$ was given for which $\bfield$ attains the conjectured bound of $\lceil p / \lceil m / 2 \rceil \rceil$; it was asked in \cite[Question~5.4]{first-beam-paper} whether, for sufficiently high $p$ and fixed $m$, these are the only representations of $\ZZ/p\ZZ$ for which $\bfield$ attains this bound. This class of representations is described by the following proposition, which serves as a lemma for Theorem~\ref{thm:gen-deg} below; the latter result also allows to conclude that, for fixed $m$ and  sufficiently high $p$, these are the only representations for which $\bfieldr$ attains the bound. (Note that this conclusion is in the direction of \cite[Question~5.4]{first-beam-paper}, but does not answer it as $\bfield$ can be higher than $\bfieldr$.) The proof that $\bfield$ behaves as promised for these representations is given in \cite{first-beam-paper} for two of the representations in the class (and it is noted that the other cases are similar), and the (very slight) modification needed to adapt the argument to $\bfieldr$ is given in \cite{blum2024degree}. For completeness, we outline those arguments here, and indicate the modifications needed for the other representations in the class.

\begin{proposition}[slightly strengthening Proposition~5.2 in \cite{first-beam-paper}]\label{prop:sharp-case}
Let $G=\ZZ/p\ZZ$ for $p$ prime. Let $m$ be a natural number $<p$. let $\kk$ be a field containing distinct $p$th roots of unity, let $\zeta$ be a primitve $p$th root of unity, and let $V=\kk^m$. Suppose the generator $[1]\in G$ acts on $V$ via the matrix $\operatorname{diag}(\zeta,\zeta^{-1},\zeta^2,\dots,\zeta^{\pm \lceil m/2\rceil})$, where the set of exponents is $\{\pm 1, \pm 2, \dots, \pm \lceil m/2 \rceil\}$ if $m$ is even, and the same set minus any one element if $m$ is odd. Then
\[
\bfield(G,V) = \bfieldr(G,V) = \left\lceil \frac{p}{\lceil m/2\rceil}\right\rceil.
\]
\end{proposition}

\begin{proof}[Proof sketch]
    In this circumstance, per \eqref{eq:lattice-equation}, $L:=L(G,V)$ is the lattice of integer solutions to an equation mod $p$ whose coefficients are $\pm 1, \dots \pm \lceil m/2\rceil$ if $m$ is even, and this same set minus any one element if $m$ is odd. It will be convenient for the proof to index the coordinates by the coefficients themselves, so we take the equation to be
    \[
    a_1 - a_{-1} + 2a_2 - 2a_{-2} + \dots \pm \lceil m / 2\rceil a_{\pm \lceil m/2 \rceil} = 0 \pmod p.
    \]
    In view of Lemma~\ref{lem:lattice-knows-all}, proving the claimed result is equivalent to showing that $L$ is  generated by its nonnegative-orthant elements of $L^1$-norm $\leq \lceil p / \lceil m/2 \rceil \rceil$, and not by the set of elements (in any orthant) of strictly lower $L^1$ norm. The main ideas of the argument are:
    \begin{enumerate}
        \item The codimension-1 sublattice $L_0$ defined by the equation
        \[
        a_1 - a_{-1} + 2a_2 - 2a_{-2} + \dots \pm \lceil m / 2\rceil a_{\pm \lceil m/2 \rceil} = 0
        \]
        is generated by its nonnegative-orthant points of $L^1$-norm $\leq 3$; in fact, it has a basis of such points.\label{step:codim-1-sublattice}
        \item There is no point of $L$ (in any orthant) that lies off of $L_0$ and has $L^1$-norm $<\lceil p / \lceil m/2\rceil \rceil$, thus $L$ is not generated by points of $L^1$-norm $< \lceil p / \lceil m/2\rceil \rceil$.\label{step:lower-bound}
        \item There {\em is} a point $L$ of $L^1$-norm exactly $\lceil p/\lceil m/2\rceil \rceil$ lying in the nonnegative orthant that, together with $L_0$, generates $L$.\label{step:upper-bound}
    \end{enumerate}
    These steps are carried out in detail in \cite[Proposition~5.2]{first-beam-paper} for the cases that $m$ is even, or that $m$ is odd and the missing coefficient is $-\lceil m/2\rceil$, except that Step~2 is only checked for points in the nonnegative orthant because \cite{first-beam-paper} only concerns $\bfield$. In \cite[Proposition~4.5]{blum2024degree} it is noted how to adapt the argument for Step~2 so that it is quantified over points in any orthant (thus the result is obtained for $\bfieldr$). We now indicate how to modify the arguments so that they work for the other possible missing coefficients in the odd-$m$ case.

    The lattice $L$ is unaffected by negating all the coefficients, so it is no loss of generality to assume the missing coefficient is negative. Let $A\in \{ 1, 2, \dots, \lceil m/2 \rceil \}$ be such that the missing coefficient is $-A$, so the set of coefficients is $S:=\{\pm 1 ,\pm 2, \dots, \pm (A-1),A,\pm (A+1),\dots,\pm \lceil m/2 \rceil\}$.

    Step~\ref{step:codim-1-sublattice} is proven by induction on the size of $S$. For each new coefficient, we find a nonnegative-orthant point of $L_0$ of degree $2$ or $3$ whose value in the corresponding new coordinate is $1$. As argued in \cite[Proposition~5.2]{first-beam-paper}, these points form a generating set for $L_0$; and it is actually a basis, by counting, since the inductive procedure yields $m-1$ points and the rank of $L_0$ is $m-1$. The argument of \cite[Proposition~5.2]{first-beam-paper} gets us as far as $S=\{\pm 1, \pm 2, \dots, \pm (A-1), A\}$. To go further, we can alternate negative and then positive coefficients, starting with the negatives. When we add a negative coefficient $J=-A-1, \dots, -\lceil m/2\rceil$, we can use the point with coordinates $a_J = a_{-J-1} = a_{1} = 1$ and the rest of the coordinates zero, as long as $A>1$ or $J<-2$. If $A=1$ and $J=-2$ we can use the point $a_J=1$, $a_1=2$ and the rest of the coordinates zero. When we add a positive coefficient $J=A+1,\dots,\lceil m/2\rceil$, we can use the point with coordinates $a_J=a_{-J}=1$ and the rest of the coordinates zero. 

    The argument for Step~\ref{step:lower-bound} given in \cite[Proposition~5.2]{first-beam-paper}, with the modification from \cite[Proposition~4.5]{blum2024degree}, works verbatim in this more general context. The key point is only that all the coefficients have absolute value at most $\lceil m/2 \rceil$. The argument for Step~\ref{step:upper-bound} in \cite[Proposition~5.2]{first-beam-paper} works verbatim too, as $S$ contains all positive numbers from $1$ to $\lceil m/2\rceil$. 
\end{proof}

We now drop the hypothesis that $G$ is abelian. For use in Section~\ref{sec:pf-of-Dspan-main}, we describe a representation-theoretic characterization of $\Dspan$ developed in \cite{bbs-derksen}. The characterization is implicit in \cite[Sections~2.3 and~2.4]{bbs-derksen} but is not spelled out, so we include a proof for completeness.  

\begin{notation}\label{not:irr}
    Let $\Irr_\kk(G)$ denote the set of isomorphism classes of finite-dimensional irreducible representations of $G$ over $\kk$. For each $\lambda\in \Irr_\kk(G)$, fix a specific representation $V_\lambda$ representing the given isomorphism class, and let $d_\lambda$ be its dimension as a $\kk$-vector space. 
\end{notation}

Because $\kk(V)$ is a $G$-representation, one can ask for the $\kk G$-module maps of $V_\lambda$ into $\kk(V)$. Observe that because $\kk(V)$ contains $\kk(V)^G$ (on which $G$ acts trivially), the set $\Hom_{\kk G}(V_\lambda,\kk(V))$ of such maps forms a $\kk(V)^G$-vector space. In particular, we can ask whether a set of $\kk G$-module embeddings of $V_\lambda$ in $\kk(V)$ is linearly dependent or independent over $\kk(V)^G$.

\begin{lemma}\label{lem:Dspan-via-counting-embeddings}
    Let $G$ be a finite group, $\kk$ an algebraically closed field of characteristic prime to $|G|$, and $V$ a finite-dimensional representation of $G$. Then $\Dspan(G,V)$ is the minimum $d$ so that, for each $\lambda \in \Irr_\kk(G)$, there are $d_\lambda$ many $\kk(V)^G$-linearly independent embeddings of $V_\lambda$ in $\kk[V]_{\leq d}\subseteq \kk(V)$.
\end{lemma}

\begin{proof}
    We will show that a finite-dimensional $G$-stable $\kk$-subspace of $\kk(V)$ spans $\kk(V)$ over $\kk(V)^G$ if and only if, for each $\lambda\in \Irr_\kk(G)$, it is able to receive $d_\lambda$ embeddings of $V_\lambda$ that are linearly independent over $\kk(V)^G$. The lemma follows from this by taking the subspace in question to be $\kk[V]_{\leq d}$.
    
    Put $\KK:=\kk(V)^G$. By the Normal Basis Theorem from Galois theory, $\kk(V)$ is isomorphic as a $G$-representation to the regular representation of $G$ over $\KK$. Because $\kk$ is algebraically closed, $V_\lambda$ is absolutely irreducible for each $\lambda$, and its multiplicity as a $\kk G$-module summand of $\kk G$ is $d_\lambda$. It follows that $\KK\otimes_\kk V_\lambda$ is irreducible over $\KK$ and its multiplicity in $\kk(V)\cong \KK G$ is also $d_\lambda$---equivalently, $\dim_\KK\Hom_{\KK G}(\KK\otimes_\kk V_\lambda,\kk(V)) = d_\lambda$---and, as $\lambda$ varies in $\Irr_\kk(G)$, the $\KK\otimes_\kk V_\lambda$ form a complete set of representatives of the isomorphism classes of irreducible representations of $G$ over $\KK$. 

    Let $W\subseteq \kk(V)$ be a finite-dimensional $G$-stable $\kk$-subspace, and let
    \[
    W = \bigoplus_{\lambda\in \Irr_\kk(G)} \bigoplus_{j=1}^{a_\lambda} W_{\lambda,j}
    \]
    be a decomposition into irreducibles, where $a_\lambda$ is the multiplicity of $\lambda$ in $W$, and each $W_{\lambda,j}$ is isomorphic to $V_\lambda$. Fix an isomorphism $\phi_{\lambda,j}:V_\lambda\rightarrow W_{\lambda,j}$ for each $\lambda\in \Irr_\kk(G)$ and each $j=1,\dots,a_\lambda$. Because $W_{\lambda,j}\subseteq W \subseteq \kk(V)$, we can view each $\phi_{\lambda,j}$ as an element of $\Hom_{\kk G}(V_\lambda, \kk(V))$. The extension-restriction adjunction gives us a natural isomorphism
    \[
    \Hom_{\KK G}(\KK\otimes_\kk V_\lambda,\kk(V))\cong \Hom_{\kk G}(V_\lambda,\kk(V)),
    \]
    so we can view each $\phi_{\lambda,j}$ as an element of the corresponding $\Hom_{\KK G}(\KK\otimes_\kk V_\lambda, \kk(V))$. There is also a canonical identification of the $\KK\otimes_\kk V_\lambda$-isotypic component $\kk(V)_\lambda$ of $\kk(V)$ with 
    \[
    \Hom_{\KK G}(\KK \otimes_\kk V_\lambda,\kk(V)) \otimes_\KK (\KK\otimes_\kk V_\lambda) \cong \Hom_{\KK G}(\KK \otimes_\kk V_\lambda,\kk(V)) \otimes_\kk V_\lambda
    \]
    via the evaluation map $f\otimes v \mapsto f(v)$. Therefore, the following are equivalent:
    \begin{enumerate}
        \item The subspaces $W_{\lambda,1},\dots,W_{\lambda,a_\lambda}$ span $\kk(V)_\lambda$ over $\KK$.
        \item Elements $\phi_{\lambda,j}(v)$, with $1\leq j \leq a_\lambda$ and $v\in V_\lambda$, span $\kk(V)_\lambda$ over $\KK$.
        \item Elements $\phi_{\lambda,j}\otimes v$, with $1\leq j \leq a_\lambda$ and $v\in V_\lambda$, span $\Hom_{\KK G}(\KK \otimes_\kk V_\lambda,\kk(V)) \otimes_\kk V_\lambda$ over $\KK$.
        \item For any fixed nonzero $v\in V_\lambda$, the elements $\phi_{\lambda,1}\otimes v,\dots,\phi_{\lambda,a_\lambda}\otimes v$ span $\Hom_{\KK G}(\KK \otimes_\kk V_\lambda,\kk(V)) \otimes_\kk V_\lambda$ over $\KK G \cong \KK\otimes_\kk \kk G$.
        \item The elements $\phi_{\lambda,1},\dots,\phi_{\lambda,a_\lambda}$ span $\Hom_{\KK G}(\KK\otimes_\kk V_\lambda,\kk(V))$ over $\KK$.
        \item There are $d_\lambda$ of the $\phi_{\lambda,1},\dots,\phi_{\lambda,a_\lambda}$ that are linearly independent over $\KK$.
    \end{enumerate} 
    Since
    \[
    \kk(V) = \bigoplus_{\lambda\in \Irr_\kk(G)} \kk(V)_\lambda,
    \]
    it follows that $W$ spans $\kk(V)$ over $\KK$ if and only if, for each $\lambda$, there exist $d_\lambda$ of the $\phi_{\lambda,j}$'s that are $\KK$-linearly independent.

    If $W$ spans $\kk(V)$ over $\KK$ then the above constructs, for each $\lambda\in \Irr_\kk(G)$,  $d_\lambda$ many $\KK$-linearly independent $\kk G$-module embeddings of $V_\lambda$ in $W$. Conversely, if we are given $d_\lambda$ many $\KK$-linearly independent $\kk G$-module embeddings of $V_\lambda$ in $W$ for each $\lambda$, then certainly they are $\kk$-linearly independent, so in the above we can choose a decomposition for $W$ into irreducibles, and isomorphisms of the $V_\lambda$ with the components of this decomposition, so that the given embeddings appear among the $\phi_{\lambda,j}$. Then, the previous paragraph gives us that $W$ spans $\kk(V)$ over $\KK$.
\end{proof}

\begin{observation}\label{obs:matrix-test-of-li}
    For given $\lambda\in \Irr_\kk(G)$ and a proposed set of $d_\lambda$ embeddings $\psi_1,\dots,\psi_{d_\lambda}$ of $V_\lambda$ into a subspace of $\kk(V)$, the linear independence condition in Lemma~\ref{lem:Dspan-via-counting-embeddings} can be tested by fixing a basis $v_1,\dots,v_{d_\lambda}$ for $V_\lambda$ and testing nonsingularity of the matrix
    \[
    M:=\left( \psi_j(v_i)\right)_{ij}
    \]
    over $\kk(V)$. 
\end{observation}

\begin{proof}
    A $\kk(V)^G$-linear relation between the $\psi_j$'s is an element in the kernel of $M$ over $\kk(V)^G$. The rowspace of $M$ is $G$-stable because it is the image of the $G$-equviariant map $\kk(V)^G\otimes_\kk V_\lambda\rightarrow \kk(V)^{d_\lambda}$ given by $v\mapsto (\psi_1(v),\dots,\psi_{d_\lambda}(v))$; thus the kernel of $M$ is also $G$-stable, and it follows from Galois descent that this kernel is nontrivial over $\kk(V)^G$ if (and only if) it is nontrivial over $\kk(V)$.   
\end{proof}

\begin{remark}
    If $G$ is abelian, then $d_\lambda=1$ for all $\lambda$, and the linear independence condition in Lemma~\ref{lem:Dspan-via-counting-embeddings} is automatically satisfied. In this case the characterization of $\Dspan$ given by Lemma~\ref{lem:Dspan-via-counting-embeddings} collapses to the characterization used in the proof of Lemma~\ref{lem:lattice-knows-all}.
\end{remark}

\section{Bounds on $\bfieldr$ and $\bfield$}\label{sec:bfieldr}

In this section, we prove the following theorem.

\begin{theorem}\label{thm:gen-deg}

    Let $G=\ZZ/p\ZZ$, $p$ an odd prime. Let $\kk$ be a field containing distinct $p$th roots of unity. Let $V$ be a finite-dimensional representation of $G$ over $\kk$, with $m$ distinct, nontrivial isotypic components.

    For fixed $m$, and $p$ sufficiently large, we have
    \[
    \bfieldr \leq \left\lceil\frac{p}{\lceil m/2\rceil}\right\rceil.
    \]
    Furthermore, equality is attained (if and) only if $V$ is one of the representations described in \cite[Question~5.4]{first-beam-paper}, up to trivial and duplicate characters.
\end{theorem}

By Lemma~\ref{lem:trivial-and-duplicate}, we can assume without loss of generality that $V$ has no trivial or duplicate characters. We make this assumption without comment throughout the rest of the secion. Thus $\dim V = m$. 

Let $L := L(G,V)\subseteq \ZZ^m\subseteq \RR^m$ be the lattice of exponent vectors of invariant Laurent monomials, as in Definition~\ref{def:L(G,V)}. Then, as discussed in Section~\ref{sec:preliminaries}, $L$ is the set of integer solutions to the equation
$$A_1a_1+\dots +A_ma_m=0\pmod p,$$
where $A_1,\dots,A_m\in \ZZ$ represent the characters occurring in $V^*$ via the correspondence \eqref{eq:integers-characters}.

By Lemma~\ref{lem:lattice-knows-all}, the statement of Theorem~\ref{thm:gen-deg} is equivalent to the following:

For fixed dimension $m$ and sufficiently large $p$, there exists a generating set for $L$ each of whose members has $L^1$ norm $\leq \left\lceil{p}/{\lceil m/2\rceil}\right\rceil$, with equality attained only in the case that $L$ is a lattice of the form contemplated in the proof of Proposition~\ref{prop:sharp-case}. In fact, we will prove something stronger:
\begin{theorem}\label{thm:gen-deg-lattice}
    For fixed dimension $m$ and sufficiently large $p$:
    \begin{itemize}
        \item There exists a basis for $L$ each of whose members has $L^1$ norm $\leq \left\lceil{p}/{\lceil m/2\rceil}\right\rceil$, and furthermore,
        \item if $L$ has no generating set contained in an $L^1$ ball of radius strictly less than $\lceil p / \lceil m/2\rceil\rceil$, then $L$ is a lattice of the form described in the proof of Proposition~\ref{prop:sharp-case}.
    \end{itemize} 
\end{theorem}

The proof of this theorem is assembled below from a number of lemmata. Because it is stated in the regime $p\gg m$, the implied constants in the Landau big-$O$ notation used throughout the proof may depend on $m$ but not on $p$.

Let $\Delta$ be the unit $L^1$ ball in $\RR^m$. Its volume is $2^m/m!$. The covolume of $L$ is its index in $\ZZ^m$, which is $p$ by Lemma~\ref{lem:index-is-|G|}. Define $\lambda_1,\dots,\lambda_m$ to be the successive minima of $\Delta$ with respect to $L$ (the definition is recalled in Section~\ref{sec:preliminaries}). Because $\Delta$ is the closed unit $L^1$ ball in $\RR^m$, each $\lambda_j$ is the smallest positive real number so that there exist $j$ linearly independent elements $\bfv_1,\dots,\bfv_j\in L$ such that
\[
\|\bfv_i\|_1 \leq \lambda_j\text{ for } i=1,\dots,j.
\]
More generally, for any $\bfv \in \RR^m$, the gauge function of $\Delta$, $f(\bfv):=\inf(r>0:\bfv\in r\Delta)$,  coincides with $\|\bfv\|_1$.

Applying Minkowski's Second Theorem (equation \eqref{eq:minkowski-2nd}), we get 
$$\lambda_1 \cdots \lambda_m \frac{2^m}{m!} \leq 2^m p,$$
and thus
$$\lambda_1 \cdots \lambda_m \leq m!p.$$
An upper bound can then be placed on $\lambda_{m-1}$.

\begin{lemma}\label{lem:minima_bound}
The successive minimum $\lambda_{m-1}$ satisfies $$\lambda_{m-1}\leq \sqrt{m!p}.$$

\end{lemma}
\begin{proof}
The minimum $L^1$ norm of any nonzero point in $L$ is at least $1$, as $L \subseteq \ZZ^m$; thus $1\leq \lambda_j$ for all $j$. Also, $\lambda_{m-1}\leq \lambda_m$. Using this, an upper bound can be placed on $\lambda_{m-1}$ by comparing the values of the other successive minima to the  lower bounds just stated, in the following manner:

\[
    1\cdot \ldots \cdot 1\cdot\lambda_{m-1}\cdot\lambda_{m-1}\leq \lambda_1 \cdot \ldots \cdot \lambda_{m-2}\cdot\lambda_{m-1}\cdot \lambda_m 
    \leq m! p.
\]
Therefore, 
\[
\lambda_{m-1}^2\leq m!p.\qedhere
\]
\end{proof}

Using this, we can construct a basis $\bfb_1,\dots,\bfb_m$ for $L$ such that all but $\bfb_m$ have $L^1$ norm small compared to $p$:

\begin{lemma}\label{lem:basis_points_exist}
There exists a basis, $\bfb_1,\bfb_2,\dots,\bfb_{m-1},\bfb_m$, for $L$ such that $\det(\bfb_1,\bfb_2,\dots,\bfb_{m-1},\bfb_m)=p$ and $\|\bfb_i\|_1\leq O(\sqrt{p})$ for $i=1,\dots, m-1$.
\end{lemma}
\begin{proof}

A theorem of Mahler \cite[Ch. 2 Sec. 10 eq. (11)]{lekkerkerker-gruber} yields the existence of a basis $\bfb_1,\dots,\bfb_m$ for $L$ such that
\begin{equation}\label{eq:mahler}
    \|\bfb_i\|_1 \leq i\lambda_i
\end{equation}
for each $i=1,\dots,m$. Because $[\ZZ^m:L]=p$, we have $\det(\bfb_1,\dots,\bfb_m)=\pm p$, and we can force it to be $+p$ by replacing any $\bfb_i$ with $-\bfb_i$ if needed. 
This gives us 
$$\|\bfb_i\|_1\leq i\lambda_i\leq (m-1)\lambda_{m-1}\leq (m-1)\sqrt{m!p}$$
for any $i=1,\dots,m-1$, with the last inequality by Lemma~\ref{lem:minima_bound}. So, for fixed $m$ and sufficiently large $p$, we get
\[
\|\bfb_i\|_1 \leq O(\sqrt{p})
\]
for any $i=1,\dots,m-1$.
\end{proof}

Fix a basis $\bfb_1,\dots,\bfb_m$ furnished by Lemma~\ref{lem:basis_points_exist} for the rest of the section.

\begin{lemma}\label{lem:cross_product_0modp}
For all $\bfw\in L$, $\det(\bfb_1,\bfb_2,\dots,\bfb_{m-1},\bfw)=0\pmod{p}$.
\end{lemma} 

\begin{proof}
The points $\bfb_1,\bfb_2,\dots,\bfb_{m-1},\bfw$ span a sublattice $L'$ of $L$. Since $L'\subseteq L\subseteq \ZZ^m$, $$[\ZZ^m : L']=[\ZZ^m:L]\cdot[L:L']=p\cdot[L:L']=0\pmod p.$$
Since $|\det(\bfb_1,\bfb_2,\dots,\bfb_{m-1},\bfw)|=[\ZZ^m : L']$, we conclude \[
\det(\bfb_1,\bfb_2,\dots,\bfb_{m-1},\bfw)=0\pmod p.\qedhere
\]
\end{proof}

Because $\det$ is a multilinear integer polynomial, $\det(\bfb_1,\bfb_2,\dots,\bfb_{m-1},\bfw)$ is a linear form in $\bfw$ with integer coefficients. Let $D_1,\dots,D_m\in \ZZ$ be these coefficients, so that
\[
D_1w_1 + \dots + D_mw_m = \det(\bfb_1,\bfb_2,\dots,\bfb_{m-1},\bfw)
\]
for all $\bfw=(w_1,\dots,w_m)\in \RR^m$.

\begin{lemma}\label{lem:cross_product_sublattice}
The lattice $\Lambda$ defined by the linear congruence condition $D_1a_1 + \dots + D_ma_m = 0 \pmod p$ contains $L$.\end{lemma} 
\begin{proof}
For all $\bfw\in L$, Lemma~\ref{lem:cross_product_0modp} tells us that $$D_1w_1 + \dots + D_mw_m =\det(\bfb_1,\bfb_2,\dots,\bfb_{m-1},\bfw)=0 \pmod p,$$ so $\bfw\in \Lambda$. Therefore, $\Lambda\supseteq L$.
\end{proof}

\begin{lemma}\label{lem:largest_coordinate_bound}
There exists some $D_j$ such that $|D_j|\geq \lceil m/2\rceil$. And if no $D_j$ strictly exceeds $ \lceil m/2\rceil$, then $L$ is a lattice of the type discussed in the proof of Proposition~\ref{prop:sharp-case}.
\end{lemma}
\begin{proof}
In view of Lemma~\ref{lem:cross_product_sublattice}, we have 
\begin{equation}\label{eq:lambda-containment}
    L \subseteq \Lambda \subseteq \ZZ^m.
\end{equation}

We know $[\ZZ^m : L]=p$ by Lemma~\ref{lem:index-is-|G|}, and $[\ZZ^m : L] = [\ZZ^m : \Lambda][\Lambda : L]$ by  \eqref{eq:lambda-containment}. Because $p$ is prime, $[\Lambda:L]$ must be $p$ or $1$. 
If $[\Lambda:L]=1$, then $\Lambda=L,$ and if $[\Lambda:L]=p$, then $\Lambda=\ZZ^m.$ We consider each case in turn.

Case 1: $\Lambda=L$. Since the coefficients $A_1,\dots,A_m$ of the equation defining $L$ are pairwise unequal and non-zero mod $p$, the coefficients of the equation defining $\Lambda$, $D_1,\dots,D_m$, are also pairwise unequal and non-zero mod $p$, as $\Lambda = L$ and this information is encoded in the lattice itself by Observation~\ref{obs:points-from-triv-and-identical}, and thus will hold true for the coefficients of any equation defining the lattice.

If there does not exist a $D_j$ such that $|D_j|> \lceil m/2\rceil$, the coefficients of the equation defining $L$ must be $$\pm1, \pm2, \dots ,\pm \left\lceil m/2\right\rceil,$$ with some element excluded in the case of odd $m$, since these are the only sets of $m$ distinct, nonzero integers all of absolute value less than $\lceil m /2 \rceil$. In this case, the $D_j$ of greatest absolute value satisfies $|D_j|= \lceil m/2\rceil$ exactly, and the lattice is of the form discussed in the proof of Proposition~\ref{prop:sharp-case}. Otherwise, there exists a $D_j$ such that $|D_j|> \lceil m/2\rceil$ strictly.

Case 2: $\Lambda=\ZZ^m$. All coefficients of $D_1a_1 + \dots + D_ma_m = 0\pmod p$ must be $0 \pmod p$. Additionally, there must be some coefficient not equal to $0$ because the points $\bfb_1,\bfb_2,\dots,\bfb_{m-1}\in L$ are linearly independent, so there exists some $\bfw\in L$ such that $\det(\bfb_1,\bfb_2,\dots,\bfb_{m-1},\bfw)=p$. This implies $D_1w_1 + \dots + D_mw_m = p,$ in which case there is some nonzero coefficient $D_j=0\pmod p$ for which $|D_j|\geq p>\lceil m/2\rceil$.
\end{proof}

Let $D^*$ be any coefficient $D_j$ satisfying $|D_j|\geq \lceil m/2\rceil$; its existence is guaranteed by Lemma~\ref{lem:largest_coordinate_bound}. The following lemma is the key point in the proof of Theorem~\ref{thm:gen-deg-lattice} (and thus Theorem~\ref{thm:gen-deg}).

\begin{lemma}\label{lem:point_in_L}
There exists a point $\bfb^*\in L$ which forms a basis with $\bfb_1, \dots , \bfb_{m-1}$, and such that 
\[
\|\bfb^*\|_1\leq \frac{p}{|D^*|}  +O(\sqrt p).
\]
\end{lemma}
\begin{proof}

The affine hyperplane \[
W:=\{\bfw\in\RR^m:\det(\bfb_1,\bfb_2,\dots,\bfb_{m-1},\bfw)=p\}
\]
contains points of $L$ (for example, $\bfb_m$). Therefore, it is partitioned into translates $\bfb+\mathcal{P}$, for $\bfb\in L\cap W$, of the $(m-1)$-dimensional parallelotope 
\[
\mathcal{P}:=\{r_1\bfb_1+\dots+r_{m-1}\bfb_{m-1}: 0 \leq r_1,\dots,r_{m-1}< 1\}.
\]
Any $\bfb\in L\cap W$ forms a basis for $L$ with $\bfb_1,\dots,\bfb_{m-1}$, because $\det(\bfb_1,\bfb_2,\dots,\bfb_{m-1},\bfb)=p$. So to prove the lemma, it suffices to find some $\bfb^*\in L\cap W$ satisfying the desired inequality.

Let $\bfu:=(0,0,\dots,p/D^*,\dots, 0,0)$, with the nonzero coordinate corresponding to the coefficient $D^*$. We have $\bfu\in W$ because
$$\det(\bfb_1,\bfb_2,\dots,\bfb_{m-1},\bfu)=D_1\cdot0 + \dots + D^*\cdot \frac p{D^*} +\dots+ D_m\cdot 0 = p.$$ Therefore, $\bfu$ is in $\bfb^*+\mathcal{P}$ for some $\bfb^*\in L\cap W$.
See Figure~\ref{fig:tiling}. Then
\[
\bfu = \bfb^* + \sum_{i=1}^{m-1}r_i^*\bfb_i
\]
with all the $r_i^*$'s satisfying $0\leq r_i^* < 1$. So, by the triangle inequality,

\[\|\bfb^*\|_1\leq\| \bfu\|_1+\sum_{i=1}^{m-1}\|\bfb_i\|_1\leq\frac p {|D^*|}+{O}(\sqrt p).\qedhere \]
\end{proof}

We are ready to assemble the proof of theorem~\ref{thm:gen-deg-lattice}.

\begin{proof}[Proof of theorem~\ref{thm:gen-deg-lattice}]

When $D^*>\lceil m/2\rceil$, $\|\bfu\|_1$ is less than $  p/\lceil m/2 \rceil$ by a fixed fraction of $p$, which eclipses $ O(\sqrt p)$ for sufficiently large $p$, so Lemma~\ref{lem:point_in_L} furnishes us with a basis in the open $L^1$ ball of radius $\lceil p / \lceil m/2\rceil \rceil$.

Additionally, for sufficiently large $p$, $\bfieldr=\|\bfb^*\|_1$.
We can always choose a $D^*>\lceil m/2\rceil$ except in the case of the lattices discussed in the proof of Proposition~\ref{prop:sharp-case}, so this gives us the bound $\bfieldr <\lceil p/{\lceil m/2 \rceil} \rceil$ for sufficiently large $p$ when $L$ is not such a lattice.

When $L$ is one of the lattices discussed in the proof of Proposition~\ref{prop:sharp-case}, then that proposition gives the equality $\bfieldr =\lceil p/{\lceil m/2 \rceil} \rceil$.

We have now obtained the desired bound
    \[
    \bfieldr \leq \left\lceil\frac{p}{\lceil m/2\rceil}\right\rceil,
    \]
with equality only in the case that $L$ is one of the lattices described in the proof of Proposition~\ref{prop:sharp-case}. This completes the proof.
\end{proof}

\begin{proof}[Proof of Theorem~\ref{thm:gen-deg}]
    Theorem~\ref{thm:gen-deg} now follows from Theorem~\ref{thm:gen-deg-lattice} by Lemma~\ref{lem:lattice-knows-all} and Lemma~\ref{lem:trivial-and-duplicate}, as discussed at the beginning of this section.
\end{proof}

\begin{figure} 
    \centering

\tdplotsetmaincoords{75}{42}
\begin{tikzpicture}
        [tdplot_main_coords,
        axis/.style={->,line width=0.4mm},
        axissegment/.style={line width=0.4mm},
        plane2/.style={opacity=.45, shading = axis, left color={rgb,255:red,62; green,63; blue,149}, right color={rgb,255:red,172; green,163; blue,255}},
        plane1/.style={opacity=.35, shading = axis, left color={rgb,255:red,64; green,124; blue,95}, right color={rgb,255:red,154; green,234; blue,205}},
        plltope1/.style={pattern={Lines},pattern color=red,dashed,opacity=.95},
        plltope2/.style={pattern={Lines},pattern color=orange,dashed,opacity=.9},
        basis/.style={-{Stealth[round]}, color={rgb,255:red,255; green,0; blue,0}, line width=0.45mm}]
        
	\draw[axissegment] (-6,0,0) -- (-17/6,0,0);
        \draw[axis] (0,0,17/4) -- (0,0,7.5) node[anchor=west]{$z$};

\foreach \point in {(-4,-2,2), (-4,2,3), (-4,6,4), (-3,-4,3), (-3,0,4), (-3,4,5), (-2,-2,5), (-2,2,6), (-1,-4,6)}
       {\node[circle, fill, opacity=.5, color={rgb,255:red,110; green,110; blue,110}, inner sep=1.45] at \point {};
}

    \draw[plane1] (3.413,-4.32,8.28) -- (1.58,6.68,8.28) -- (-5.753,6.68,-2.72) -- (-3.92,-4.32,-2.72) -- cycle;

        \foreach \point in {(-4,-1,-2), (-4,3,-1), (-3,-3,-1), (-3,1,0), (-3,5,1), (-2,-1,1), (-2,3,2), (-1,-3,2), (-1,1,3), (-1,5,4), (0,-1,4), (0,3,5), (1,-3,5), (1,1,6)}
       {\node[circle, fill, color={rgb,255:red,80; green,110; blue,30}, inner sep=1.5] at \point {};
}

    \draw[axissegment] (-17/6,0,0) -- (0,0,0);
    \draw[axissegment] (0,-6,0) -- (0,0,0);
    \draw[axissegment] (0,0,0) -- (0,0,17/4);

\foreach \point in {(-17/6,0,0)}
       {\node[circle, fill, color=blue, inner sep=1.5] at \point {};
}

\draw[plltope2] (-4,-1,-2) -- (-2,-1,1) -- (-1,1,3) -- (-3,1,0) -- cycle;
\draw[line width=0.3mm] (-4,-1,-2) -- (-2,-1,1);
\draw[line width=0.3mm] (-4,-1,-2) -- (-3,1,0);

    \draw[plane2] (5.333,-4,7) -- (3.5,7,7) -- (-3.833,7,-4) -- (-2,-4,-4) -- cycle;

        \foreach \point in {(-3,2,-4), (-3,6,-3), (-2,0,-3), (-2,4,-2), (-1,-2,-2), (-1,2,-1), (-1,6,0), (0,0,0), (0,4,1), (1,-2,1), (1,2,2), (1,6,3), (2,0,3), (2,4,4), (3,-2,4), (3,2,5), (4,0,6)}
       {\node[circle, fill, color={rgb,255:red,80; green,30; blue,110}, inner sep=1.5] at \point {};
}

    \draw[axis] (0,0,0) -- (8,0,0) node[anchor=west]{$x$};
    \draw[axis] (0,0,0) -- (0,8,0) node[anchor=west]{$y$};
    \draw[axissegment] (0,0,-5) -- (0,0,0);

    \draw[plltope1] (0,0,0) -- (2,0,3) -- (3,2,5) -- (1,2,2) -- cycle;

    \draw[basis] (0,0,0) -- (2,0,3);
    \draw[basis] (0,0,0) -- (1,2,2);

\foreach \point in {(0,1,-4), (0,5,-3), (1,-1,-3), (1,3,-2), (2,-3,-2), (2,1,-1), (2,5,0), (3,-1,0), (3,3,1), (4,-3,1), (4,1,2), (4,5,3), (5,-1,3), (5,3,4), (6,-3,4), (6,1,5)}
       {\node[circle, fill, opacity=.5, color={rgb,255:red,80; green,80; blue,80}, inner sep=1.5] at \point {};
}

\foreach \point in {(2,6,-4), (3,0,-4), (3,4,-3), (4,-2,-3), (4,2,-2), (4,6,-1), (5,-4,-2), (5,0,-1), (5,4,0), (6,-2,0), (6,2,1)}
       {\node[circle, fill, opacity=.5, color={rgb,255:red,50; green,50; blue,50}, inner sep=1.5] at \point {};
}

\draw (-17/6,0,0) node [scale=0.8,outer sep=2.5,anchor= south west]{$\bfu$};
\draw (1,2,2) node [scale=0.8,outer sep=2.5,anchor= west]{$\bfb_1$};
\draw (2,0,3) node [scale=0.8,outer sep=2.5,anchor= south]{$\bfb_2$};
\draw (2.5,1,4) node [inner sep=5, anchor = south] {$\mathcal{P}$};
\draw (-4,-1,-2) node [scale=0.8,outer sep=2.5,anchor= west] {$\bfb^*$};
\draw (1.5,4,7.5) node [scale=0.8] {$W$};

\end{tikzpicture}
\caption{Illustration of the proof of Lemma~\ref{lem:point_in_L} in the case $m=3$. The affine hyperplane $W$ is the set $\{\bfw\in \RR^3:\det(\bfb_1,\bfb_2,\bfw)=p\}$. It is partitioned into translates of the parallelogram $\mathcal P$ spanned by $\bfb_1,\bfb_2$ by points of $L$ lying in $W$. One of these translates, $\bfb^*+\mathcal P$ with $\bfb^*\in L\cap W$, contains the point $\bfu$ where $W$ intersects an axis with controlled $L^1$ norm---specifically, $\|\bfu\|_1 = |p/D^*|$, where $D^*$ is some coefficient of the equation for $L$ that exceeds $\lceil m/2 \rceil$ in absolute value. The point $\bfb^*$ of $L$ forms a basis with $\bfb_1,\bfb_2$ and has controlled $L^1$ norm.}
\label{fig:tiling}
\end{figure}
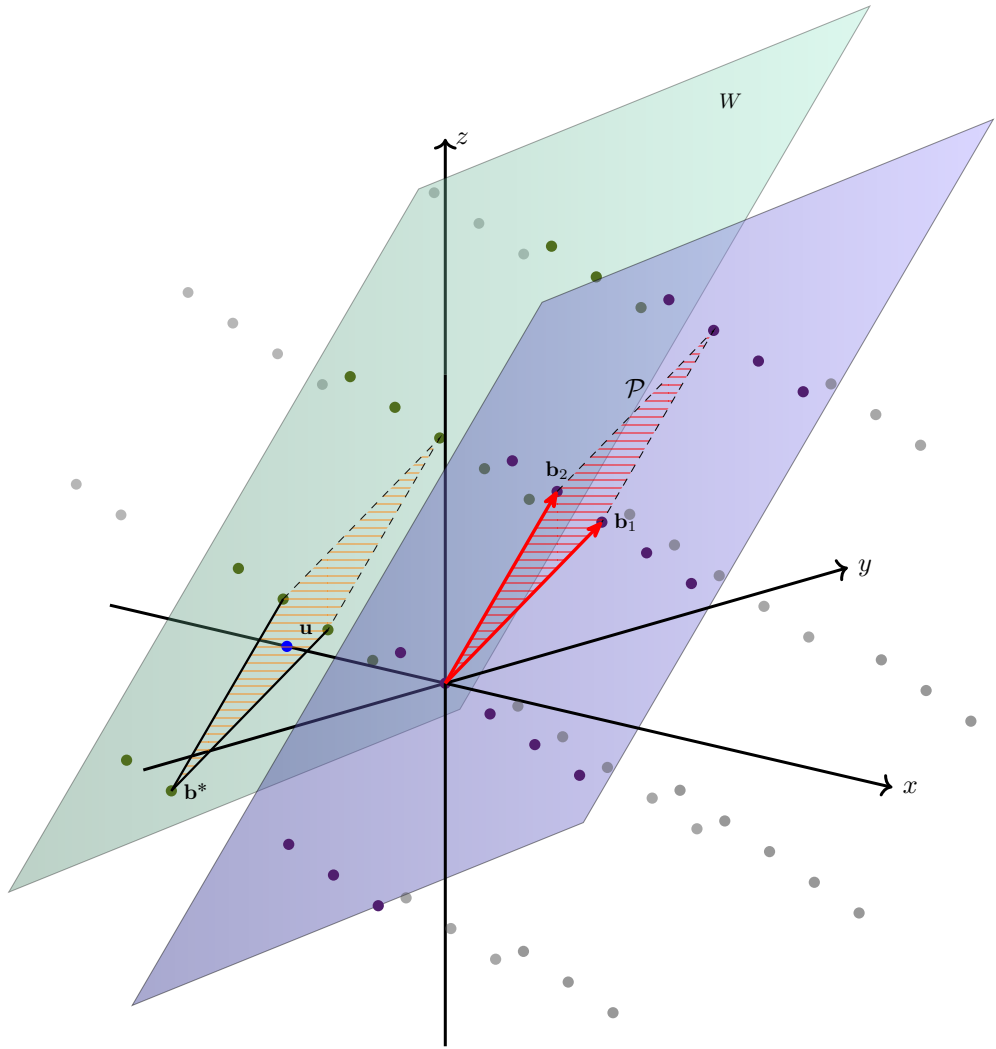

The bound proven in Theorem~\ref{thm:gen-deg} falls short of the conjectured bound \cite[Conjecture~5.1]{first-beam-paper} in two respects: it concerns $\bfieldr$ rather than $\bfield$, and holds only for $p$ in the regime $p\gg m$. We now address the first of these defects, upgrading the result to $\bfield$ if we assume that the isotypic components in the representation $V$ come in dual pairs; this is automatic if, for example, $\kk=\RR$ (the context relevant ot the signal processing application discussed in the introduction). 

\begin{theorem}\label{thm:real-reps-bfield}
Let $G=\ZZ/p\ZZ$, $p$ an odd prime, and let $\kk$ be a field of characteristic coprime to $p$. Let $V$ be a representation of such that the $m$ nontrivial isotypic components of $V$ (after base-changing to a field containing distinct $p$th roots of unity) come in dual pairs. Then, for $p\gg m$, we have
\[
\bfield(G,V) \leq \left\lceil \frac{p}{m/2} \right \rceil.
\]
In particular, \cite[Conjecture~5.1]{blum2024degree} holds under the hypotheses that $\kk=\RR$ and $p\gg m$.
\end{theorem}

\begin{proof}
By Lemmas~\ref{lem:lattice-knows-all} and \ref{lem:trivial-and-duplicate}, and the fact that $\bfield$ is insensitive to base change \cite[Lemma~2.1]{first-beam-paper}, it is equivalent to show the following claim: Let $L$ be a lattice defined by an equation $A_1a_1+\dots+A_ma_m=0\pmod p$ with coefficients distinct and nonzero mod $p$, such that for all $i=1,\dots,m$, there exists $j$ with $A_i+A_j=0\pmod p$. Then, if $p$ is sufficiently large depending only on $m$, $L$ contains a generating set lying in the intersection of the nonnegative orthant with the $L^1$-norm ball of radius $\left\lceil{2p}/m \right \rceil$. 

By Theorem~\ref{thm:gen-deg-lattice}, under these hypotheses there is a generating set for $L$ lying in the $L^1$-norm ball of radius $\left\lceil{2p}/m \right \rceil$, but not necessarily lying in the first orthant. However, for every pair $i,j$ with $A_i+A_j=0\pmod p$, $L$ also contains the point $\bfu_{ij}:=\bfe_i+\bfe_j$, which lies in the nonnegative orthant and has $L^1$-norm $2 < \lceil 2p/m\rceil$. Given any point $\bfa=(a_1,\dots,a_m)\in L$, let $c_{ij}(\bfa)=-\min(a_i,a_j,0)$. Then $\bfa+c_{ij}(\bfa)\bfu_{ij}$ is in $L$, is nonnegative in the $i$th and $j$th coordinates, and satisfies
\[
\|\bfa+c_{ij}(\bfa)\bfu_{ij}\|_1 \leq \|\bfa\|_1.
\]
Furthermore, $\bfa$ is generated by $\bfu_{ij}$ and $\bfa + c_{ij}(\bfa)\bfu_{ij}$. We can repeat this process with every inverse pair of coefficients, yielding the point
\[
\bfa' := \bfa + \sum_{\{i,j\}}c_{ij}(\bfa)\bfu_{ij},
\]
where the sum is taken over all the pairs of indices $\{i,j\}$ with $A_i+A_j=0\pmod p$. The hypothesis is that these pairs partition the full index set, so that $\bfa'$ lies in the nonnegative orthant, and again $\|\bfa'\|_1\leq \|\bfa\|_1$ and $\bfa$ is generated by $\bfa'$ and the $\bfu_{ij}$'s. 

This device allows the known generating set $\bfa_1,\dots,\bfa_m$ for $L$ satisfying the desired bound to be pushed into the nonnegative orthant, yielding a set $\bfa_1',\dots,\bfa_m'$ that, together with the $\bfu_{ij}$'s, generates $L$ and satisfies the desired bound.
\end{proof}

\begin{remark}
    The possibility of deducing Theorem~\ref{thm:real-reps-bfield} from Theorem~\ref{thm:gen-deg} was suggested by \cite[Proposition~3.1]{blum2024degree}, which states that $\bfield=\bfieldr$ if $\kk$ is a field whose only roots of unity are $\pm 1$ (such as $\RR$). Something like the proof given here is needed, however; Theorem~\ref{thm:real-reps-bfield} does not follow straight from Theorem~\ref{thm:gen-deg} by way of \cite[Proposition~3.1]{blum2024degree}   because Theorem~\ref{thm:gen-deg} requires that $\kk$ contain $p$th roots of unity and base-change can lower $\bfieldr$.
\end{remark}

\begin{remark}
    The proof of Theorem~\ref{thm:gen-deg} via Theorem~\ref{thm:gen-deg-lattice} actually finds a basis for the lattice $L(G,V)$ satisfying the desired bounds, not just a generating set. Applying $\exp$ (as in Notation~\ref{not:log-exp}), we get an {\em algebraically independent} generating set for $\kk(V)^G$ satisfying a certain degree bound, provided the ground field contains enough roots of unity. From this point of view, Theorem~\ref{thm:gen-deg} may be regarded as a result in the ``constructive approach to Noether's problem" \cite{kemper1996constructive}, which seeks explicit algebraically independent generating sets for fields that have them. On the other hand, Theorem~\ref{thm:real-reps-bfield} only provides a generating set for $L(G,V)$, not a basis, due to the use of the $\bfu_{ij}$'s, so the resulting generating set for $\kk(V)^G$ is not algebraically independent.
\end{remark}

The proof of Theorem~\ref{thm:gen-deg} requires the primality of $p$. Particularly, without this hypothesis, the proof would fail in Lemma~\ref{lem:largest_coordinate_bound}, as $L \subseteq \Lambda \subseteq \ZZ^m$ does not necessarily imply $[\Lambda:L]=1$ or $[\Lambda:L]=p$ if $p$ is composite. The question was posed in \cite[Question~5.1]{first-beam-paper}  whether the conjectured bound in \cite[Conjecture~5.1]{first-beam-paper} holds if $p$ is composite. We give a counterexample, showing that it does not hold for $\bfieldr$, and therefore also does not hold for $\bfield$ since $\bfieldr\leq\bfield$.
\begin{proposition}
For a lattice $L^*$ in dimension $m=5$ defined by an equation of the form $$a_1 + \left(\frac{n}{2}-1\right)a_2 + \left(\frac{n}{2}\right)a_3 +\left(\frac{n}{2} + 1\right)a_4 + \left(n-1\right)a_5 = 0 \pmod n,$$ where $n$ is even, we have
    \[
    \bfield(L^*)\geq \frac{n}{2}>\left\lceil\frac{n}{\lceil m/2\rceil}\right\rceil.
    \]
\end{proposition}
\begin{proof}
The lattice $L^*$ contains the linearly independent points $$\bfb_1:=(0,1,0,1,0),\;\bfb_2:=(0,0,1,1,1),\;\bfb_3:=(1,1,1,0,0),\;\bfb_4:=(0,0,2,0,0)$$ which span the hyperplane $$W:=\{\bfw\in\RR^5:\det(\bfb_1,\bfb_2,\bfb_3,\bfb_4,\bfw)=0\}.$$

If we let $\mathbf{D}\in \ZZ^5$ be the unique point such that $\mathbf{D}\cdot \bfw = \det(\bfb_1,\bfb_2,\bfb_3,\bfb_4,\bfw)$ for all $\bfw\in \RR^5$, then $\mathbf{D}=(2,-2,0,2,-2)$. For all $\bfw\in L^*$, we have $\mathbf{D}\cdot\bfw=0\pmod n$, with $\mathbf{D}\cdot\bfw=0$ only in the case that $\bfw\in L^*\cap W$. Any generating set for $L^*$ must contain at least one point $\bfw\notin W$. We can call such a point $\bfw^*$.

We now know that $|\mathbf{D}\cdot \bfw^*|\geq n$. Since $\mathbf{D}=(2,-2,0,2,-2),$ we can write this as $|2\cdot w^*_1 -2\cdot w^*_2+2\cdot w^*_4-2\cdot w^*_5|\geq n$. Rearrangement leads to $|w^*_1 -w^*_2+w^*_4-w^*_5|\geq n/2$. Since $\|\bfw^*\|_1\geq|w^*_1 -w^*_2+w^*_4-w^*_5|$, it follows that 
$$\bfieldr\geq\|\bfw^*\|_1\geq n/2.$$
Therefore,
\[
 \bfield\geq\bfieldr\geq n/2 > \left\lceil\frac{n}{\lceil m/2\rceil}\right\rceil.\pighere
\]
\end{proof}

\begin{remark}
The desired bound on $\bfieldr$ was proven for fixed dimension $m$ and sufficiently large $p$ by making use of the bound $\|\bfb_i\|_1 \leq O(\sqrt{p})$ established in Lemma \ref{lem:basis_points_exist}. This asymptotic bound does not indicate how large $p$ has to be for the proof to go through. It is possible to make Theorem~\ref{thm:gen-deg} effective via the following more precise bounds on the $L^1$ norms of the points $\bfb_1,\dots,\bfb_{m-1}$ used in the proof, which can be extracted from Minkowski's Second Theorem without a great deal of additional work. (These are relevant to finding the regime in which the proof strategy succeeds, but are not expected to be a comment on $\bfield$ itself since the bound in Theorem~\ref{thm:gen-deg} is conjectured to hold without the hypothesis $p\gg m$.)

\begin{proposition}\label{prop:precise_basis_bounds}
The basis $\bfb_1,\bfb_2,\dots,\bfb_{m-1},\bfb_m$ of Lemma \ref{lem:basis_points_exist} can be chosen so that $$\|\bfb_i\|_1 \leq i\left(\frac{m! p}{2^{i-1}}\right)^{\frac{1}{m-i+1}}$$ for $i=1,\dots,\lfloor\frac{m}{2} \rfloor$, and 
$$\|\bfb_i\|_1 \leq i\left({\frac{m! p}{2^{\lfloor \frac{m}{2}\rfloor} 3^{i-1-\lfloor\frac{m}{2} \rfloor}}}\right)^{\frac 1 {m-i+1}}$$ for $i=\lfloor\frac{m}{2} \rfloor+1,\dots, m-1$

\end{proposition}
\begin{proof}
As discussed in the remark on page $11$ of \cite{blum2024degree}, using the arguments made in the proofs of \cite[Proposition~3.5]{first-beam-paper} and \cite[Proposition~4.3]{blum2024degree}, we can obtain the result that $L$ has no points of $L^1$ norm $1$, and no more than $\lfloor m/2\rfloor$ points of $L^1$ norm $2$. Therefore, $\lambda_i\geq2$ for $i=1,\dots,\lfloor\frac{m}{2} \rfloor$, and $\lambda_i\geq3$ for $i=\lfloor\frac{m}{2} \rfloor+1,\dots, m-1$.

Using this, when $i=1,\dots,\lfloor\frac{m}{2} \rfloor$, we have the following factor-by-factor inequality:
$$2^{i-1}\lambda_i^{m-i+1}=2\cdot\ldots\cdot2\cdot\lambda_i\cdot \ldots\cdot\lambda_i\leq \lambda_1\cdot\ldots\cdot\lambda_{i-1}\cdot\lambda_i\cdot\ldots\cdot\lambda_m\leq m!p.$$
This yields the upper bound $$\lambda_i \leq \left(\frac{m! p}{2^{i-1}}\right)^{\frac{1}{m-i+1}}.$$
Since $\|\bfb_i\|_1 \leq i\lambda_i$ by \eqref{eq:mahler}, for $i=1,\dots,\lfloor\frac{m}{2} \rfloor$ we have
\[
\|\bfb_i\|_1 \leq i\left(\frac{m! p}{2^{i-1}}\right)^{\frac{1}{m-i+1}}.
\]
Similarly, in the case that $i=\lfloor\frac{m}{2} \rfloor+1,\dots, m-1$, we can utilize the inequalities $2\leq \lambda_1,\dots,\lambda_{\lfloor m/2\rfloor}$, $3\leq \lambda_{\lfloor m/2\rfloor+1}, \dots, \lambda_{i-1}$, and $\lambda_i\leq \lambda_{i+1},\dots,\lambda_m$ to obtain the inequality
$$2^{\lfloor\frac{m}{2}\rfloor}3^{i-1-\lfloor\frac{m}{2}\rfloor}\lambda_i^{m-i+1}\leq m!p.$$
Rearrangement leads to 
$$\lambda_i^{m-i+1}\leq \frac{m!p}{2^{\lfloor\frac{m}{2}\rfloor}3^{i-1-\lfloor\frac{m}{2}\rfloor}}.$$

Therefore, 
\[
\lambda_i \leq \left({\frac{m! p}{2^{\lfloor \frac{m}{2}\rfloor} 3^{i-1-\lfloor\frac{m}{2}\rfloor}}}\right)^{\frac 1 {m-i+1}}.
\]
Since $\|\bfb_i\|_1 \leq i\lambda_i$ by \eqref{eq:mahler}, for $i=\lfloor\frac{m}{2} \rfloor+1,\dots, m-1$ we have
\[
\|\bfb_i\|_1 \leq i\left({\frac{m! p}{2^{\lfloor \frac{m}{2}\rfloor} 3^{i-1-\lfloor\frac{m}{2}\rfloor}}}\right)^{\frac 1 {m-i+1}}.\qedhere
\]
\end{proof}

\end{remark}

\section{Bound on $\Dspan$}\label{sec:HRD}

The aim of this section is to prove the following theorem on $\Dspan$:

\begin{theorem}\label{thm:Dspan-main}
    Let $G$ be a finite group. Let $V$ be a faithful representation of $G$ over a field $\kk$ of characteristic prime to the order of $G$. Then
    \[
    \Dspan(G,V) \leq |G|/2
    \]
    except in the special case where $V$ is a scalar representation or a direct sum of a scalar and a trivial representation.
\end{theorem}

In the exceptional case that $V$ is a scalar or scalar plus trivial representation of a cyclic group, we have $\Dspan(G,V)=|G|-1$. Before discussing the proof, we give some examples where equality is realized in Theorem~\ref{thm:Dspan-main}, showing that it is sharp. One of them is used as a lemma, so we explicate the examples without relying on Theorem~\ref{thm:Dspan-main}.

\begin{example}[Cyclic group]\label{ex:sharp-cyclic}
    Let $G=C_n$, the cyclic group of order $n$, with $n$ even. 
    \begin{itemize}
        \item Let $\zeta = e^{2\pi i / n}\in \CC$, and let $V=\CC^2$, with a generator $g$ for $G$ acting by the matrix
    \[
    \begin{pmatrix}\zeta & \\ & \zeta^{-1}
    \end{pmatrix}.
    \]
    If $x_1,x_2$ are coordinate functions, then a minimal-degree polynomial basis for $\CC(V)$ over $\CC(V)^G$ is given by 
    \[
    x_2^{n/2-1},\dots,x_2,1,x_1\dots,x_1^{n/2},
    \]
    so $\Dspan(G,V)=n/2=|G|/2$. This can be seen by considering $L(G,V)$, for which the defining condition \eqref{eq:lattice-equation} takes the form
    \[
    a_1 - a_2 = 0\pmod n.
    \]
    The points $(0,n/2-1),\dots,(0,1),(0,0),(1,0),\dots,(n/2,0)$ ($=\log (x_2^{n/2-1}),\dots,\log(x_1^{n/2})$) represent every coset of $L(G,V)$ in $\ZZ^2$, and are each the (or, in the case of $(n/2,0)$, one of two) nonnegative point(s) of minimal $L^1$ norm in their respective cosets.
        \item With the same setup as in the previous bullet, let $V=\CC^2$, but this time with $g$ acting by the matrix
        \[
        \begin{pmatrix}
            \zeta & \\ & \zeta^2
        \end{pmatrix}.
        \]
        By an analysis of $L(G,V)$ similar to the above, it can be shown that a minimal-degree polynomial basis for $\CC(V)$ over $\CC(V)^G$ is given by
        \[
        x_1^ix_2^j,\; 0\leq i \leq 1,\; 0\leq j \leq n/2-1,
        \]
        so again $\Dspan(G,V) = n/2$.
    \end{itemize}

\end{example}

\begin{example}[Dihedral group]\label{ex:sharp-dihedral}
    Let $G=D_{2n}$, the dihedral group of order $2n$, and let it act on $V=\CC^2$ in its standard reflection representation. Then $G$ is a reflection group of type $I_2(n)$; it is well known that its degrees are $2,n$ (e.g., \cite[Table~3.1]{humphreys1992reflection}). So the Hilbert series of the coinvariant algebra $\CC[V]_G$ is
    \[
    H(\CC[V]^G,t) = (1+t)(1+\dots+t^{n-1}),
    \]
    e.g., by \cite[\S~23--1, Theorem~A]{kane2001reflection}. This is a polynomial of degree $(2-1)+(n-1)=n$, so $\Dspan(G,V)\leq \topdeg(G,V)=n$. By the same token, the lowest-degree appearance of the sign character in $\CC[V]$ is also in degree $(2-1)+(n-1)=n$, e.g., by \cite[Proposition~3.13]{humphreys1992reflection}. 
    Recall from the introduction that $\Dreg$ is the minimum degree $d$ such that $\CC[V]_{\leq d}$ contains a copy of the regular representation of $G$. Since the regular representation cannot be complete without the sign character of $G$, it follows that $n \leq \Dreg(G,V)$. Meanwhile, it follows from the Noether-Deuring Theorem and the Normal Basis Theorem that $\Dreg \leq \Dspan$; see \cite[Proposition~4.1]{bbs-derksen}.  Combining all this, we have
    \[
    n \leq \Dreg \leq \Dspan\leq\topdeg=n=|G|/2.
    \]
    The same argument shows in general that, for any finite reflection group, $\Dreg=\Dspan=\topdeg=\sum_j(d_j-1)$, the number of reflections (where the $d_j$ are the degrees); but it happens that for dihedral groups, exactly half of the elements in the group are reflections.
\end{example}

Theorem~\ref{thm:Dspan-main} is proven by a case analysis. The main case is when $V$ is cyclic and contains a faithful representation of dimension $\leq 2$. In this case, it will be deduced from the following theorem about lattices.

\begin{theorem}\label{thm:HRD-main}
    If  $L \subseteq \ZZ^2$ is a sublattice of index $n$, and it is not of the form  $\{(a_1,a_2)\in\ZZ^2:A_1a_1+A_2a_2=0\pmod n\}$ with either $A_1,A_2 \in \ZZ$ congruent to each other mod $n$ or one of them zero mod $n$, 
         then 
         \[
         \Dspanl \leq n/2.
         \]
\end{theorem}

The next subsection proves Theorem~\ref{thm:HRD-main}. The following subsection completes the proof of Theorem~\ref{thm:HRD-main}, using Theorem~\ref{thm:HRD-main} for the main case.

\subsection{Proof of Theorem~\ref{thm:HRD-main}}\label{sec:pf-of-HRD-main}

The proof of Theorem~\ref{thm:HRD-main} is assembled below from several lemmata which carry the bulk of the work. We first set up the needed conventions and definitions.

\begin{convention}
    We use $a_1,a_2$ as general coordinate functions on $\RR^2$ throughout the below; thus, for example, {\em the line $a_2=-a_1$} refers to the set $\{(a_1,a_2)\in \RR^2: a_2=-a_1\}$, and {\em above the line $a_2=-a_1$} refers to points $(a_1,a_2)\in \RR^2$ satisfying $a_2>-a_1$.
    
    We use {\em first quadrant}, {\em second quadrant}, and {\em fourth quadrant} in their conventional meanings on $\RR^2$, i.e., the first quadrant is the set $\{(a_1,a_2)\in \RR^2: a_1,a_2\geq 0\}$, the second quadrant is $\{(a_1,a_2)\in \RR^2: a_1\leq 0, a_2\geq 0\}$, and the fourth quadrant is $\{(a_1,a_2)\in \RR^2: a_1\geq 0, a_2\leq 0\}$. 

    We endow $\RR^2$ with the product partial order induced by the usual total order on $\RR$. Thus $\bfa \leq \bfb$ means $a_1\leq b_1$ and $a_2\leq b_2$, where $\bfa=(a_1,a_2)$ and $\bfb=(b_1,b_2)$.
\end{convention}

\begin{definition}\label{def:weight}
Given a point $\bfa=(a_1,a_2)$, we write
\[
\wt(\bfa) := a_1+a_2
\]
and call this the {\em weight} of $\bfa$.
\end{definition}

\begin{remark}\label{rem:wteqlnorm}
    The terminology reflects the fact that $\wt(\bfa)$ is the weight of the scalar action of $\kk^\times$ on $\bfx^\bfa:=\exp(\bfa)=x_1^{a_1}x_2^{a_2}$, i.e., for $\xi\in \kk^\times$ and $v\in V=\kk^2$, viewing $\bfx^\bfa$ as a polynomial function on $V$ we have $\bfx^\bfa(\xi v) = \xi^{\wt(\bfa)} \bfx^\bfa(v)$. We will use below that $\wt:\ZZ^2\rightarrow\ZZ$ is additive, and that $\wt(\bfa)=\|\bfa\|_1$ if $\bfa$ is in the first quadrant.
\end{remark}

\begin{lemma}\label{lem:bite-lemma}
    Let $\bfa$ be a point of $L$ with $\wt(\bfa)>0$. Let $\bfb\in \ZZ^2$ be an integer point satisfying $\bfb \geq \bfa$. Then there exists a point $\bfc\in \NN^2$ in the first quadrant such that 
    \[
    \bfc = \bfb \pmod L
    \]
    and
    \[
    \wt(\bfb) > \|\bfc\|_1.
    \]
\end{lemma}

\begin{proof}
    The point $\bfb - \bfa$ is such a $\bfc$, as follows.

     Because $\wt(\bfa) > 0$ and $\wt$ is additive, $\bfb - \bfa$ satisfies the inequality $\wt(\bfb)>\wt(\bfb - \bfa)$. As their difference, $\bfa$, is a point of $L$, both $\bfb$ and $\bfb- \bfa$ belong to the same class mod $L$.  The point $\bfb - \bfa$ is in the first quadrant as $b_1 \geq a_1$ and $b_2 \geq a_2$ so $b_1  - a_1, b_2 - a_2 \geq 0$. Therefore $\|\bfb - \bfa\|_1=\wt(\bfb-\bfa)<\wt(\bfb)$, and we can conclude.
\end{proof}

\begin{definition}\label{def:blob}
    Let $B$ consist of a set of coset representatives for $L$ in $\ZZ^2$ such that each $\bfb$ of $B$ is of minimum $L^1$ norm in $(\bfb + L)\cap \NN^2$.
\end{definition}

\begin{observation}\label{obs:Dspan-B}
    By Lemma~\ref{lem:lattice-knows-all}, part~\ref{item:Dspan-from-lattice} we have
    \[
    \Dspan(L)=\max_{\bfb\in B}(\|\bfb\|_1).
    \]
\end{observation}

\begin{lemma}\label{lem:blob-lemma}
     If $\bfa$ is a point of $L$ with $\wt(\bfa)>0$, and $\bfb \in B$, then $\bfa \nleq \bfb$.
\end{lemma}
\begin{proof}
   If $\bfb \geq \bfa$ then by Lemma~\ref{lem:bite-lemma} there exists some point $\bfc$ in $\NN^2$ such that $\wt(\bfb) > \|\bfc\|_1$ and  $\bfc = \bfb\pmod L$. Meanwhile, $\bfb\in B\subseteq \NN^2$, so $\wt(\bfb)=\|\bfb\|_1$, and thus $\|\bfb\|_1>\|\bfc\|_1$. This is a contradiction by the minimality requirement in the definition of $B$ (Definition~\ref{def:blob}). 
\end{proof}

\begin{lemma}\label{lem:staircase-lemma}
    Suppose $\bfa^1, \dots, \bfa^r$ are points of $L$ with $\wt(\bfa^j)>0$ for each $j=1,\dots,r$. Let 
    \[
    \widehat B := \{\bfu \in \NN^2: \bfu\ngeq \bfa^j\text{ for each }j=1,\dots,r\}.
    \]
    Then:
    \begin{enumerate}
        \item By reordering the $\bfa^j$'s and dropping some if necessary, we can assume that their first coordinates are strictly increasing and their second coordinates are strictly decreasing, without changing $\widehat B$.\label{cond:monotonic}
        \item If the conditions in statement~\ref{cond:monotonic} hold, then
        \begin{enumerate}
            \item $\widehat B$ is finite if and only if both $a^1_1\leq 0$ and $a^r_2\leq 0$, and\label{part:Bhat-finite}
            \item if $\widehat B$ is finite, then
            \[
            \Dspan(L) \leq \max_{j=2,\dots,r}(a^{j}_1+a^{j-1}_2-2).
            \]\label{part:staircase-corners}
        \end{enumerate}
    \end{enumerate}
\end{lemma} 

\begin{proof}
    Although the statements are about the lattice $L$, parts of the argument are most readily expressed in the language of monomials and Laurent monomials. Per Notation~\ref{not:LM}, let $\mathcal{LM}$ denote the group (under multiplication) of Laurent monomials in $x_1,x_2$, and per Notation~\ref{not:log-exp}, for $\bfa\in \ZZ^2$ let $\exp(\bfa):=x_1^{a_1}x_2^{a_2}\in \mathcal{LM}$ denote the Laurent monomial with exponent vector $\bfa=(a_1,a_2)$. The  canonical embedding of the polynomial algebra $\kk[x_1,x_2]$ into the group algebra $\kk[\mathcal{LM}]$ makes the latter a module over the former. Let
    \[
    M:=\langle \exp(\bfa^1),\dots,\exp(\bfa^r)\rangle_{\kk[x_1,x_2]}
    \]
    be the $\kk[x_1,x_2]$-submodule generated by the Laurent monomials with exponent vectors $\bfa^1,\dots,\bfa^r$. The situation is illustrated in Figure~\ref{fig:illustration-of-staircase-lemma}. Then 
    \[
    I:=M\cap \kk[x_1,x_2]
    \]
    is a monomial ideal in $\kk[x_1,x_2]$, and $\widehat B$ is precisely the set of exponent vectors of the monomials of $\kk[x_1,x_2]$ that do not belong to it. Thus the set of (canonical images of the) monomials $\exp(\widehat B)$ in $\kk[x_1,x_2]/I$ forms a $\kk$-basis for this ring.
    
    Passing to a minimal module generating set for $M$ and then lexicographically ordering it, we guarantee that $a^1_1<\dots<a^r_1$ and $a^1_2>\dots>a^r_2$ as in the conditions in statement~\ref{cond:monotonic}, without changing $M$ at all, and therefore without changing $I$ or $\widehat B$. We assume going forward that these conditions are met.

    To prove statement~\ref{part:Bhat-finite}, observe that because $a^1_1 \leq a^j_1$ for all $j$,  $a_1^1\leq 0$ if and only if some point $(0,a_2)$ satisfies $\bfa^1\leq (0,a_2)$, if and only if some power of $x_2$ is in $M$ and thus in $I$; likewise, because $a^r_2\leq a^j_2$ for all $j$, $a^r_2\leq 0$ if and only if some power of $x_1$ is in $M$ and thus in $I$. Therefore, the algebra $\kk[x_1,x_2]/I$ has finite $\kk$-dimension, or equivalently $\widehat B$ is a finite set, if and only if both $a_1^1\leq 0$ and $a^r_2\leq 0$.

    To prove statement~\ref{part:staircase-corners}, we first observe that the $B$ of Definition~\ref{def:blob} is contained in $\widehat B$, as follows. Let $\bfb \in B$. Because $\bfa^j$ is in $L$ with $\wt(\bfa^j) > 0$ for $j = 1, \dots, r$, the hypothesis on $\bfa$ in Lemma \ref{lem:blob-lemma} applies to each $\bfa^j$. Therefore, $\bfb \ngeq \bfa^j$ for each $j$. Also, $\bfb\in B \subseteq  \NN^2$. These are the conditions that define $\widehat B$, so $\bfb \in \widehat B$. Therefore, $B \subseteq \widehat B$. 

    Thus, 
    \[
    \Dspan(L)= \max_{\bfb\in B}(\|\bfb\|_1) \leq \max_{\bfu\in \widehat B}(\|\bfu\|_1),
    \]
    where the equality is Observation~\ref{obs:Dspan-B} and the inequality follows from $B\subseteq \widehat B$. Statement~\ref{part:staircase-corners} is then proven by establishing that, assuming finiteness of $\widehat B$, we must have
    \begin{equation}\label{eq:topdeg-I}
    \max_{\bfu\in \widehat B}(\|\bfu\|_1)\leq\max_{j=2,\dots,r}(a^{j}_1+a^{j-1}_2-2).
    \end{equation}
    This holds because the maximum $L^1$ norm in $\widehat B$ must be attained by one of the points
    \[
    (a^j_1-1,a^{j-1}_2-1),
    \]
    with norm $a^j_1+a^{j-1}_2-2$, for some $j=2,\dots,r$.\footnote{The inequality sign in \eqref{eq:topdeg-I} is necessary to hedge against the possibility that there exist some $j$'s for which $a^j_1-1$ or $a^{j-1}_2-1$  are negative.} This is probably most readily apprehended by looking at the {\em staircase diagram} for the ideal $I$, see \cite[Section~3.1]{miller2005combinatorial}, or for the module $M$, see Figure~\ref{fig:illustration-of-staircase-lemma}; but we include a proof for completeness.

    Let $\bfu=(u_1,u_2)\in\widehat B$. Finiteness of $\widehat B$ gives $a^1_1\leq 0$ and $a^r_2\leq 0$ by statement~\ref{part:Bhat-finite}.  Because $a^r_2\leq 0$, if also $u_1\geq a^r_1$ then $\bfu\geq \bfa^r$, contrary to the construction of $\widehat B$. Thus $u_1 < a^r_1$. Also, $a^1_1\leq 0 \leq u_1$, and so there is a unique $j=2,\dots,r$ with $u_1$ lying in the half-open interval $a^{j-1}_1 \leq u_1 < a^j_1$. Fix this $j$ going forward.
    
    If $u_1\neq a^j_1-1$, then $u_1+1$ also lies in the same interval, so it satisfies all the same inequalities in relation to the $a^j_1$'s that $u_1$ does. Therefore, $(u_1+1,u_2)$ is still in $\widehat B$, and $\bfu$ does not maximize the $L^1$ norm on $\widehat B$. 
    
    If $u_1 = a^j_1-1\geq a^{j-1}_1$, then first note that it is impossible for $u_2$ to be $\geq a^{j-1}_2$, as this would imply $\bfu \geq \bfa^{j-1}$, contrary to the construction of $\widehat B$. So $u_2 \leq a^{j-1}_2-1$. If the inequality is strict, then $u_2+1$ still satisfies $u_2+1<a^{j-1}_2$, and therefore $u_2+1 < a^\ell_2$ for any $\ell\leq j-1$. Meanwhile, obviously $u_1$ still satisfies $u_1<a^j_1$ and therefore $u_1<a^\ell_1$ for any $\ell \geq j$. Thus $(u_1,u_2+1)$ is not $\geq \bfa^\ell$ for any $\ell$, so it still in $\widehat B$, and $\bfu$ does not maximize the $L^1$-norm inside $\widehat B$. Thus $\bfu$ cannot be a maximizer of $\|\bfu\|_1$ on $\widehat B$ unless both $u_1=a^j_1-1$ and $u_2=a^{j-1}_2-1$. 
    
    Since $\widehat B$ is finite, there does exist a maximizer, and the inequality~\eqref{eq:topdeg-I} follows, completing the proof.  
\end{proof}

\begin{figure}
    \begin{center}
    \begin{tikzpicture}[scale=0.5]
    \draw[->] (-3,0) -- (9,0) node[right] {$a_1$};
    \draw[->] (0,-3) -- (0,8) node[above] {$a_2$};
    
    \draw[<-] (-2.5,8) -- (-2.5,5.5);
    \draw (-2.5,5.5) -- (4.5,5.5) -- (4.5,2.5) -- (6.5,2.5) -- (6.5,-2.5);
    \draw[->] (6.5,-2.5) -- (9,-2.5);
    
    \filldraw (-2,6) circle (2pt) node[above right] {$\mathbf{a}^1$};
    \filldraw (5,3) circle (2pt) node[above right] {$\mathbf{a}^2$};
    \filldraw (7,-2) circle (2pt) node[above right] {$\mathbf{a}^3$};
    
    \filldraw (4,5) circle (2pt);
    \filldraw (6,2) circle (2pt);
    
    \node at (2.5,3) {$\widehat{B}$};

    \node at (6,7) {$M$};
    
    \node[below left] at (0,0) {$\mathbf{0}$};
\end{tikzpicture}
    \end{center}
    \caption{Illustration of Lemma~\ref{lem:staircase-lemma}. The maximum $L^1$ norm in $\widehat B$ must be attained at one of the unlabeled points.}
    \label{fig:illustration-of-staircase-lemma}
\end{figure}
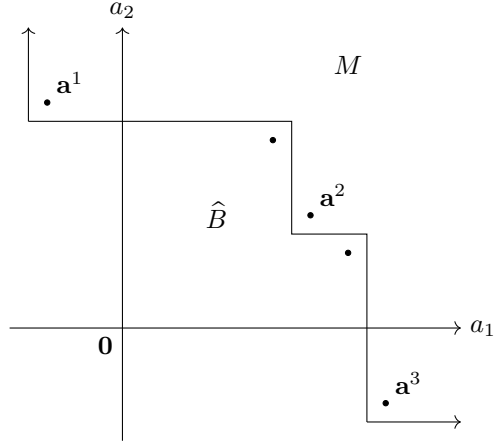

\begin{definition}
Let $\bfH=(H_1,H_2)$ be the point of $L$ with the lowest $2$nd coordinate among those in the second quadrant and above the line $a_2 = -a_1$.  If there are multiple such points, let it be the one with the least negative $1$st coordinate. 
\end{definition}

\begin{definition}
Let $\bfE=(E_1,E_2)$ be the point of $L$ with the lowest $1$st coordinate among those in the fourth quadrant and above the line $a_2 = -a_1$. If there are multiple such points, then among those, let it be the one with the least negative $2$nd coordinate.
\end{definition}

\begin{lemma}\label{lem:H-E-basis} Either $\bfE$ and $\bfH$ form a basis for $L$, or else the weights of $\bfH$ and $\bfE$ are equal and there exists at least one point of $L$ different from $\bfH$ or $\bfE$ in the first quadrant and on the line segment between $\bfH$ and $\bfE$. Except possibly for this segment, there are no nonzero points of $L$ in the triangle with vertices $\mathbf{0},\bfH,\bfE$. 
\end{lemma}

\begin{proof}
Suppose that $\bfH$ and $\bfE$ do not form a basis for $L$. Then there must exist some nonzero point $\bfD\in L$ different from $\bfE$ or $\bfH$ in the parallelogram $\{c_1\bfE+c_2\bfH:0\leq c_1,c_2<1\}$, as the latter forms a fundamental domain for $L$ in the plane.\footnote{In fact, there must exist such a $\bfD$ in the closed triangle with vertices $\mathbf 0, \bfH, \bfE$, and the present proof can be somewhat shortened by exploiting this; however, the present method reveals some symmetry obscured by this shortcut.}

We claim that any such $\bfD$ must lie in the first quadrant and must satisfy $\wt(\bfD)=\wt(\bfE)=\wt(\bfH)$. The lemma statement follows immediately from this. The argument is illustrated in Figure~\ref{fig:proof-idea-H-E-basis}, and proceeds in three steps.

If $\bfD=(D_1,D_2)$ satisfies $\wt(D)>0$, $D_1\leq 0$, and $D_2<H_2$, or $\wt(D)>0$, $D_2=H_2$, and $H_1< D_1 \leq 0$, then it contradicts the construction of $\bfH$. Let $\mathcal{H}$ be the region defined by these inequalities. Similarly, if $\bfD$ satisfies $\wt(\bfD)>0$, $D_2\leq 0$, and $D_1<E_1$, or $\wt(\bfD)>0$, $D_1=E_1$, and $E_2< D_2 \leq 0$, then it contradicts the construction of $\bfE$; let $\mathcal{E}$ be the region defined by these inequalities. The regions $\mathcal E$ and $\mathcal H$ are depicted by blue triangles in the leftmost image in Figure~\ref{fig:proof-idea-H-E-basis}, with dotted and solid lines indicating excluded and included parts of the boundaries. The pink regions in this image are the images of the blue regions under the transformation $\bfu \mapsto \bfH + \bfE - \bfu$; thus $\bfD$ cannot lie in them without $\bfH + \bfE - \bfD$ contradicting the construction of $\bfH$ or $\bfE$. The complement in the parallelogram $\{c_1\bfE+c_2\bfH:0\leq c_1,c_2<1\}$ of the set
\[
\mathcal H \cup \mathcal E \cup (\bfH + \bfE - \mathcal H) \cup (\bfH + \bfE - \mathcal E)
\]
is the interior of the rectagular region $\{\bfu\in \RR^2: \mathbf{0}\leq \bfu \leq \bfE + \bfH\}$. 
It follows that any nonzero $\bfD\in L$ in the parallelogram $\{c_1\bfE+c_2\bfH:0\leq c_1,c_2<1\}$ different from either $\bfE$ or $\bfH$ must in fact lie in this rectangle interior.

Next, the center image in Figure~\ref{fig:proof-idea-H-E-basis} depicts the regions $\mathcal E + \bfH$ and $\bfH - \mathcal H$ overlaid on the rectangular region $\{\bfu\in \RR^2: \mathbf{0}\leq \bfu \leq \bfE + \bfH\}$ (the latter is drawn with dotted boundary, as we now know $\bfD$ must lie in its interior). If $\bfD \in \mathcal E + \bfH$, then $\bfD - \bfH$ contradicts the construction of $\bfE$; if $\bfD \in \bfH - \mathcal H$, then $\bfH - \bfD$ contradicts the construction of $\bfH$. The complement in $\{\bfu\in \RR^2: \mathbf{0}\leq \bfu \leq \bfE + \bfH\}$ of
\[
(\mathcal E + \bfH) \cup (\bfH - \mathcal H)
\]
is the intersection of $\{\bfu\in \RR^2: \mathbf{0}\leq \bfu \leq \bfE + \bfH\}$ with the line $\{\bfu\in \RR^2: \wt(\bfu) = \wt(\bfH)\}$. Thus any nonzero $\bfD\in L$ in $\{c_1\bfE+c_2\bfH:0\leq c_1,c_2<1\}$ different from either $\bfE$ or $\bfH$ must in fact satisfy $\wt(\bfD)=\wt(\bfH)$ and lie in the first quadrant.

Finally, the rightmost image in Figure~\ref{fig:proof-idea-H-E-basis} depicts the regions $\mathcal H + \bfE$ and $\bfE - \mathcal E$, overlaid on the rectangular region $\{\bfu\in \RR^2: \mathbf{0}\leq \bfu \leq \bfE + \bfH\}$. If $\bfD\in L$ is in $\mathcal H + \bfE$, then $\bfD - \bfE$ contradicts the construction of $\bfH$; meanwhile if $\bfD \in L$ is in $\bfE - \mathcal E$, then $\bfE - \bfD$ contradicts the construction of $\bfE$. The complement of
\[
(\mathcal H + \bfE) \cup (\bfE - \mathcal E)
\]
in $\{\bfu\in \RR^2: \mathbf{0}\leq \bfu \leq \bfE + \bfH\}$ is the intersection of the latter with the line $\{\bfu\in \RR^2: \wt(\bfu) = \wt(\bfE)\}$. So any $\bfD\in L$ in $\{c_1\bfE+c_2\bfH:0\leq c_1,c_2<1\}$ different from either $\bfE$ or $\bfH$ must in fact satisfy $\wt(\bfD)=\wt(\bfE)$ and lie in the first quadrant. This completes the proof of the claim and thus the lemma.
\end{proof}

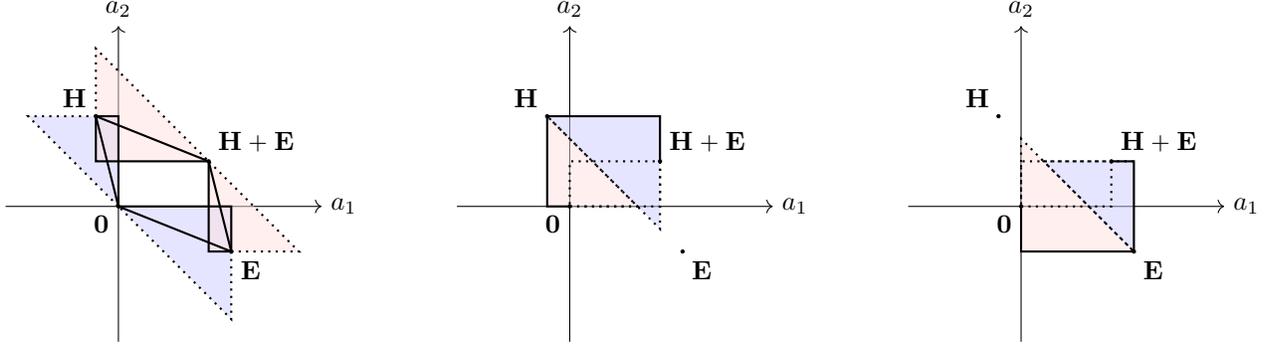
\begin{figure}
\begin{center}
\begin{tikzpicture}[scale=0.3]

\begin{scope}[xshift=0cm]
\draw[->] (-5,0) -- (9,0) node[right] {$a_1$};
\draw[->] (0,-6) -- (0,8) node[above] {$a_2$};

\fill[blue!20, opacity=0.5] (0,0) -- (5,0) -- (5,-5) -- cycle;

\fill[blue!20, opacity=0.5] (0,0) -- (0,4) -- (-4,4) -- cycle;

\fill[pink!50, opacity=0.5] (4,2) -- (-1,2) -- (-1,7) -- cycle;

\fill[pink!50, opacity=0.5] (4,2) -- (4,-2) -- (8,-2) -- cycle;

\draw[thick] (0,0) -- (5,0) -- (5,-2);
\draw[thick, dotted] (5,-2) -- (5,-5) -- (0,0);

\draw[thick] (0,0) -- (0,4) -- (-1,4);
\draw[thick, dotted] (-1,4) -- (-4,4) -- (0,0);

\draw[thick] (4,2) -- (-1,2) -- (-1,4);
\draw[thick, dotted] (-1,4) -- (-1,7) -- (4,2);

\draw[thick] (4,2) -- (4,-2) -- (5,-2);
\draw[thick, dotted] (5,-2) -- (8,-2) -- (4,2);

\draw[thick] (0,0) -- (5,-2) -- (4,2) -- (-1,4) -- cycle;

\filldraw[black] (0,0) circle (2pt) node[below left] {$\mathbf{0}$};
\filldraw[black] (5,-2) circle (2pt) node[below right] {$\mathbf{E}$};
\filldraw[black] (4,2) circle (2pt) node[above right] {$\mathbf{H+E}$};
\filldraw[black] (-1,4) circle (2pt) node[above left] {$\mathbf{H}$};
\end{scope}

\begin{scope}[xshift=20cm]
\draw[->] (-5,0) -- (9,0) node[right] {$a_1$};
\draw[->] (0,-6) -- (0,8) node[above] {$a_2$};

\fill[blue!20, opacity=0.5] (-1,4) -- (4,4) -- (4,-1) -- cycle;

\fill[pink!50, opacity=0.5] (-1,4) -- (-1,0) -- (3,0) -- cycle;

\draw[thick] (-1,4) -- (4,4) -- (4,2);
\draw[thick, dotted] (4,2) -- (4,-1) -- (-1,4);

\draw[thick] (-1,4) -- (-1,0) -- (0,0);
\draw[thick, dotted] (0,0) -- (3,0) -- (-1,4);

\draw[thick, dotted] (0,0) -- (4,0) -- (4,2) -- (0,2) -- cycle;

\filldraw[black] (0,0) circle (2pt) node[below left] {$\mathbf{0}$};
\filldraw[black] (5,-2) circle (2pt) node[below right] {$\mathbf{E}$};
\filldraw[black] (4,2) circle (2pt) node[above right] {$\mathbf{H+E}$};
\filldraw[black] (-1,4) circle (2pt) node[above left] {$\mathbf{H}$};
\end{scope}

\begin{scope}[xshift=40cm]
\draw[->] (-5,0) -- (9,0) node[right] {$a_1$};
\draw[->] (0,-6) -- (0,8) node[above] {$a_2$};

\fill[blue!20, opacity=0.5] (5,-2) -- (5,2) -- (1,2) -- cycle;

\fill[pink!50, opacity=0.5] (5,-2) -- (0,-2) -- (0,3) -- cycle;

\draw[thick] (5,-2) -- (5,2) -- (4,2);
\draw[thick, dotted] (5,2) -- (1,2) -- (5,-2);

\draw[thick] (5,-2) -- (0,-2) -- (0,0);
\draw[thick, dotted] (0,0) -- (0,3) -- (5,-2);

\draw[thick, dotted] (0,0) -- (4,0) -- (4,2) -- (0,2) -- cycle;

\filldraw[black] (0,0) circle (2pt) node[below left] {$\mathbf{0}$};
\filldraw[black] (5,-2) circle (2pt) node[below right] {$\mathbf{E}$};
\filldraw[black] (4,2) circle (2pt) node[above right] {$\mathbf{H+E}$};
\filldraw[black] (-1,4) circle (2pt) node[above left] {$\mathbf{H}$};
\end{scope}

\end{tikzpicture}
\end{center}
\caption{Illustration of the proof of Lemma~\ref{lem:H-E-basis}. Left: Any $\bfD\in L$ that lies in the blue triangles $\mathcal H$ or $\mathcal E$ (including the solid parts of their boundaries) contradicts the construction of $\bfH$ or $\bfE$. If there is a $\bfD\in L$ in the pink triangles $\bfH + \bfE - \mathcal E$ or $\bfH + \bfE - \mathcal H$, then $\bfH + \bfE -\bfD$ contradicts the construction of $\bfE$ or $\bfH$. Thus, any $\bfD\in L$ in the parallelogram $\{c_1\bfE+c_2\bfH:0\leq c_1,c_2<1\}$ must in fact lie in the interior of the rectangle $\{\bfu:\mathbf{0} \leq \bfu \leq \bfH + \bfE\}$. Middle: If there is a $\bfD\in L$ in the blue triangle $\mathcal E + \bfH$, then $\bfD - \bfH$ contradicts the construction of $\bfE$. If there is a $\bfD\in L$ in the pink triangle $\bfH - \mathcal H$, then $\bfH - \bfD$ contradicts the construction of $\bfH$. These regions cover the interior of the rectangle $\{\bfu:\mathbf{0}\leq \bfu \leq \bfH + \bfE\}$ except for the line $\{\bfu:\wt(\bfu) = \wt(\bfH)\}$. Right: If there is a $\bfD\in L$ in the blue triangle $\mathcal H + \bfE$, then $\bfD-\bfE$ contradicts the construction of $\bfH$. If there is a $\bfD\in L$ in the pink triangle $\bfE - \mathcal E$, then $\bfE - \bfD$ contradicts the construction of $\bfE$. These regions cover the interior of the rectangle $\{\bfu:\mathbf{0}\leq \bfu \leq \bfH + \bfE\}$ except for the line $\{\bfu:\wt(\bfu) = \wt(\bfE)\}$. In these images, it happens that $\wt(\bfH)=\wt(\bfE)$.} 
\label{fig:proof-idea-H-E-basis}
\end{figure}

\begin{lemma}\label{lem:if-basis}
    If $\bfH$ and $\bfE$ form a basis for $L$, then 
        \[
        \Dspanl \leq \max(E_1+ E_2 + H_2 -2,E_1 + H_1+ H_2 -2)) \leq n/2.
        \]
\end{lemma}
\begin{proof}
The first inequality in the lemma statement follows from applying Lemma~\ref{lem:staircase-lemma} with $\bfa^1 = \bfH$, $\bfa^2 = \bfH + \bfE$, and $\bfa^3=\bfE$, noting that $a_1^1=H_1\leq 0$ and $a_2^3 = E_2 \leq 0$. The work is to prove the second inequality. It suffices to prove that
\begin{equation*}
E_1+E_2 + H_2 - 2 \leq n/2,
\end{equation*}
because we can then obtain $E_1+H_1+H_2-2\leq n/2$ by reversing the axes and the roles of $\bfE$ and $\bfH$ in the proof.

We have the following inequalities:

\begin{equation}\label{eq:E+H-doesn't-contradict-H}
    E_1 + H_1 \geq 0
\end{equation}
because if $E_1+H_1<0$ then $\bfE + \bfH$ would contradict the construction of $\bfH$;

\begin{equation}\label{eq:Weight-positive}
    H_1 + H_2 \geq 1\text{ and }E_1 + E_2 \geq 1
\end{equation}
because the points $\bfH$ and $\bfE$ are above the line $a_2 = -a_1$; and

\begin{equation}\label{eq:no-triv-coeffs}
    E_1 \geq 2 \text{ and } H_2 \geq 2
\end{equation}
because by Observation~\ref{obs:points-from-triv-and-identical-2}, $L$ does not contain $(1,0)$ or $(0,1)$. Because $\bfE$ and $\bfH$ form a positively-oriented basis for $L$, and $[\ZZ^2:L]=n$ by hypothesis, we also have 
\begin{equation}\label{eq:determinant}
    E_1H_2 - E_2H_1 = n.
\end{equation}

We make the following variable change:
\begin{align*}
    X &= E_1 \\
    Y &= E_1 + E_2 \\
    Z &= E_1 + E_2 + H_2.
\end{align*}
The goal becomes to prove that 
\[
Z-2\leq n/2.
\]

In this language, \eqref{eq:no-triv-coeffs} yields $X \geq 2$ and $Z \geq 2+Y$, and \eqref{eq:Weight-positive} yields $Y \geq 1$.  Multiplying \eqref{eq:E+H-doesn't-contradict-H} through by $E_2$ (reversing the inequality because $E_2\leq 0$ by construction of $\bfE$), combining with \eqref{eq:determinant}, and applying the variable change, we get $X(Z-X) - n \leq 0$ and therefore
\begin{equation}
        Z \leq \frac{n}{X} + X.
\end{equation}
Note that $n/X + X$ is a convex function of $X$, and it equals $n/2 + 2$ at $X = 2$ and $X=n/2$. Therefore, for  $2 \leq X \leq n/2$, the inequality $Z -2 \leq n/2$ holds. Since $X \geq 2$, the lemma is proven with the exception of the case where $X > n/2$. For aesthetic reasons we handle the latter case under the slightly weaker hypothesis that $X\geq n/2$, which we assume going forward.

Multiplying the left inequality in \eqref{eq:Weight-positive} by $E_2$ (which reverses it), combining with \eqref{eq:determinant}, and applying the variable change, we get $(Z-Y)Y \leq Y-X+n$ and therefore
\begin{equation}\label{eq:x>n/2-case}
    Z \leq 1 + Y + \frac{n-X}{Y} \leq 1 + Y + \frac{n/2}{Y},
\end{equation}
with the second inequality because $X \geq n/2$ in the present case. Additionally, combining the first inequality in \eqref{eq:x>n/2-case} with $2+Y \leq Z$, we get
\[
2+Y \leq 1 + Y + \frac{n-X}{Y},
\]
and therefore
\[
Y \leq n-X
\]
(because $Y$ is positive). Therefore, $Y \leq n/2$. Since $Y \geq 1$, and the right side of \eqref{eq:x>n/2-case} is a convex function of $Y$ that equals $n/2 + 2$ at $Y = 1$ and $Y = n/2$, from \eqref{eq:x>n/2-case} we have $Z \leq  n/2 +2$. Therefore, $Z-2 \leq n/2$ when $X \geq n/2$ as well, and the lemma is proven. 
\end{proof}

The following elementary inequality is used to handle the case that $\bfH, \bfE$ do not form a basis for $L$. 

\begin{lemma}\label{lem:alphabeta}
Let $\alpha$ and $\beta$ be real numbers $\geq 2$. Then $\alpha + \beta -2 \leq \alpha \beta/2$.
\end{lemma}
\begin{proof}
Let $\alpha = 2 + a$ and $\beta = 2 + b$ where $a,b \geq 0$. Then
\[
    \frac{\alpha\beta}{2} - (\alpha+\beta - 2) \geq \frac{(2+a)(2+b)}{2} - ((2+a)+(2+b)-2) =  \frac{ab}{2},
\]
which is clearly nonnegative.
\end{proof}

\begin{lemma}\label{lem:if-not-basis}
    When  $\bfE$ and $\bfH$  do not form a basis of $\bf L$, 
    the inequality $\Dspanl \leq n/2$ is still satisfied.
\end{lemma}
\begin{proof}
Let $\bfa^1,\dots, \bfa^r$ be the set of points of $L$ on the line segment between $\bfa^1 = \bfH$ and $\bfa^r = \bfE$, ordered so that their 1st coordinates are increasing. By Lemma~\ref{lem:H-E-basis}, $\bfE$ and $\bfH$ have equal weight, thus $\wt(\bfH) = \wt(\bfa^j) = \wt(\bfE)$ for each $j$. From the definitions of $\bfE$ and $\bfH$, we know that all points $\bfa^j$ with $j=2,\dots, r-1$ will be in the first quadrant.

Because $\bfa^1=\bfH$ has a nonpositive first coordinate, and $\bfa^r=\bfE$ has a nonpositive second coordinate, and $\bfa^1,\dots,\bfa^r$ are ordered with increasing first and decreasing second coordinates, Lemma \ref{lem:staircase-lemma} gives us that 
\begin{equation}\label{eq:Dspan-from-staircase-corners}
        \Dspan(L) \leq \max_{j=2,\dots,r}(a^{j}_1+a^{j-1}_2-2).
\end{equation}
Because $L$ is a lattice of integer points, there exists a natural number $F$ such that $F = a_1^{j} - a_1^{j-1}$ and $-F=a_2^{j} - a_2^{j-1}$ for any $j = 2, \dots, r$. Therefore, 
\[
a^{j}_1+a^{j-1}_2-2 = a_1^{j-1}+F + a_2^{j-1} - 2 = \wt(\bfa^{j-1})+F - 2
\]
for any $j=2,\dots,r$. Since all $\bfa^j$ with $j=2,\dots, r$ have equal weight, the right side will have the same value for any $j$. Combining this with \eqref{eq:Dspan-from-staircase-corners}, we find that
\begin{equation}\label{eq:Dspan-from-a-single-corner}
\Dspan(L) \leq \wt(\bfa^{j-1})+F - 2
\end{equation}
for any $j=2,\dots,r$.

By Lemma \ref{lem:H-E-basis}, we know that $r\geq3$. Take some $j$ satisfying $3\leq j \leq r$. Then $2\leq j-1 \leq r-1$, therefore $\bfa^{j-1}$ is in the first quadrant. 
For this choice of $j$, we now verify the hypotheses of Lemma~\ref{lem:alphabeta} for $\alpha = \wt(\bfa^{j-1})$ and $\beta = F$, which we will then apply to the right side of \eqref{eq:Dspan-from-a-single-corner}.

First, $F \geq 2$ because if $F=1$ then $\bfa^j-\bfa^{j-1}=(1,-1)\in L$, but the point $(1,-1)$ cannot be contained in $L$ by Observation~\ref{obs:points-from-triv-and-identical-2}. Second, $\wt(\bfa^{j-1}) \geq 2$ because the only points in the first quadrant with weight $1$ are $(1,0)$ and $(0,1)$, but Observation~\ref{obs:points-from-triv-and-identical-2} implies that these points are not in $L$ either. Thus Lemma~\ref{lem:alphabeta} applies with $\alpha = \wt(\bfa^{j-1})$ and $\beta = F$. Combining this with \eqref{eq:Dspan-from-a-single-corner}, we get
\begin{equation}\label{eq:Dspan-nearly-there}
\Dspanl \leq \wt(\bfa^{j-1})F/2.
\end{equation}

Finally we claim that $\bfa^{j-1},\bfa^j$ form a basis of $L$. Indeed, by construction of $\bfa^{j-1},\bfa^j$ there are no points of $L$ on the segment between them, and by Lemma~\ref{lem:H-E-basis}, there are no nonzero points of $L$ anywhere else in the closed triangle with vertices $\mathbf{0}, \bfa^{j-1},\bfa^j$. It follows (e.g. by \cite[Chapter~1, Section~3, Theorem~4]{lekkerkerker-gruber}) that they are a basis.

The determinant of this basis equals the index of $L$:
\[n=(a^{j-1}_1 +F)a^{j-1}_2-a^{j-1}_1(a^{j-1}_2 -F) = \wt(\bfa^{j-1})F. \]
Therefore, \eqref{eq:Dspan-nearly-there} becomes the inequality
\[
\Dspanl \leq n/2
\]
as in the lemma statement.
\end{proof}

\begin{proof}[Proof of Theorem~\ref{thm:HRD-main}]
    Either $\bfH$ and $\bfE$ form a basis for $L$, or they do not; thus we obtain Theorem~\ref{thm:HRD-main} by combining Lemma~\ref{lem:if-basis} and Lemma~\ref{lem:if-not-basis}.
\end{proof}

\subsection{Proof of Theorem~\ref{thm:Dspan-main} from Theorem~\ref{thm:HRD-main}}\label{sec:pf-of-Dspan-main}

This section is devoted to proving Theorem~\ref{thm:Dspan-main}. The proof is a case analysis, and Theorem~\ref{thm:HRD-main} handles the hardest case. But we require two other lemmata for other cases.

The {\em dicyclic group of order $4n$}, $n$ an integer $\geq 2$, is the finite group abstractly presented by
\[
\langle \sigma,\tau : \sigma^{2n}=1,\; \sigma^n=\tau^2,\; \tau\sigma\tau^{-1}\rangle.
\]
It is denoted $Dic_{4n}$. A concrete realization is given by the following faithful representation over an algebraically closed field $\kk$ of characteristic not dividing $2n$:
\begin{equation}\label{eq:dicyclic}
\sigma\mapsto \begin{pmatrix}
    \zeta & 0 \\ 0 & \zeta^{-1}
\end{pmatrix},\quad
\tau\mapsto \begin{pmatrix}
    0 & -1 \\
    1 & 0
\end{pmatrix},
\end{equation}
where $\zeta\in \kk$ is a fixed primitive $2n$th root of unity. The center of $Dic_{4n}$ is generated by the involution $\sigma^n=\tau^2$; the quotient by the center is isomorphic to the dihedral group of order $2n$. The commutator subgroup of $Dic_{4n}$ is generated by $\sigma^2$; the abelianization is isomorphic to $C_2\times C_2$ or $C_4$, depending on whether $n$ is even or odd.

\begin{lemma}\label{lem:dicyclic}
    Let $G$ be the dicyclic group of order $4n$, $n>1$, and let $\kk$ be an algebraically closed field of characteristic not dividing $2n$. Let $V$ be the two-dimensional representation of $G$ over $\kk$ given by \eqref{eq:dicyclic}. Then $\Dspan(G,V)=n+1$.
\end{lemma}

\begin{proof}
    Let $x,y$ be coordinate functions on $V$, chosen so that $\sigma$ and $\tau$ act by $\sigma x= \zeta x$, $\sigma y = \zeta^{-1}y$, $\tau x = y$, $\tau y = -x$ (possible because $V$ is self-dual).

    Let $\Irr_\kk(G)$ be the set of isomorphism classes of irreducible representations of $G$ over $\kk$. For any $\lambda\in \Irr_\kk(G)$, let $W_\lambda$ be some representation that realizes the class $\lambda$, and let $d_\lambda$ be the dimension of $W_\lambda$ as a $\kk$-space. From Lemma~\ref{lem:Dspan-via-counting-embeddings}, $\Dspan$ is equal to the minimum degree $d$ such that, for each $\lambda\in \Irr_\kk(G)$, the $\kk$-space $\kk[V]_{\leq d}$ of polynomials of degree $\leq d$, viewed as a representation of $G$, contains $d_\lambda$ different embeddings of $W_\lambda$ that are linearly independent over $\kk(V)^G$.

    For $G$ the dicyclic group of order $4n$, there are $n+3$ classes in $\Irr_\kk(G)$, all of dimension 1 or 2:
    \begin{enumerate}
        \item The trivial representation.\label{item:trivial-rep}
        \item The one-dimensional representation defined by $\sigma \mapsto 1$, $\tau\mapsto -1$.\label{item:tau-sign-rep}
        \item $n-1$ two-dimensional representations, falling into two classes, as follows. Let $\zeta\in \kk$ be a primitive $2n$th root of unity.
        \begin{enumerate}
            \item For odd $j$ from $1$ to $n-1$, a faithful representation
            \[
            \sigma\mapsto \begin{pmatrix}
            \zeta^j & 0 \\ 
            0 & \zeta^{-j}
            \end{pmatrix},\quad
            \tau\mapsto \begin{pmatrix}
            0 & -1 \\
            1 & 0
            \end{pmatrix}.
            \]
        \label{item:dicyclic-2d-reps}
        \item For even $j$ from $1$ to $n-1$, a representation
            \[
            \sigma\mapsto \begin{pmatrix}
            \zeta^j & 0 \\ 
            0 & \zeta^{-j}
            \end{pmatrix},\quad
            \tau\mapsto \begin{pmatrix}
            0 & 1 \\
            1 & 0
            \end{pmatrix}
            \]
        that factors through the quotient of $G$ by its center, which is a dihedral group of order $2n$.\label{item:dihedral-2d-reps}
        \end{enumerate}
        \item There are two other one-dimensional representations, whose descriptions depend on the parity of $n$:
        \begin{enumerate}
        \item If $n$ is even, then there is a one-dimensional representation with kernel $\langle \sigma^2,\tau\rangle$ sending $\sigma\mapsto -1$, and another one with kernel $\langle \sigma^2, \sigma\tau \rangle$ sending $\sigma,\tau\mapsto -1$.\label{item:even-n-V}
        \item If $n$ is odd, then there are two one-dimensional representations with kernel $\langle \sigma^2\rangle$, sending $\tau \mapsto i,-i$, where $i$ is some primitive fourth root of unity in $\kk$.\label{item:odd-n-Z4}
        \end{enumerate}
    \end{enumerate}
    We need to check that a representation $W_\lambda$ of each of these types embeds $d_\lambda$-many $\kk(V)^G$-linearly independent times in $\kk[V]_{\leq n+1}$. (For the $d_\lambda = 1$ cases---items~\ref{item:trivial-rep}, \ref{item:tau-sign-rep}, \ref{item:even-n-V}, and \ref{item:odd-n-Z4} above---linear independence is automatically satisfied.) And we need to check that $n+1$ is the minimal $d$ so that this holds for $\kk[V]_{\leq d}$.

    We check each member of $\Irr_\kk(G)$ in turn. The trivial representation embeds in $\kk[V]_{\leq n+1}$ as $\kk[V]_0\cong \kk$. The one-dimensional representation $\sigma\mapsto 1$, $\tau\mapsto -1$ embeds in degree $2<n+1$ as the subspace $\langle xy\rangle_\kk$. The remaining two one-dimensional representations embed in the degree-$n$ subspace $\langle x^n,y^n\rangle_\kk$ whether $n$ is even or odd: if $n$ is odd, then the two representations described in item~\ref{item:odd-n-Z4} above embed as the subspaces $\langle x^n - iy^n\rangle_\kk$ and $\langle x^n + iy^n\rangle_\kk$, while if $n$ is even, the representations described in item~\ref{item:even-n-V} embed as $\langle x^n + y^n \rangle_\kk$ and $\langle x^n - y^n \rangle_\kk$. This covers the one-dimensional irreducibles.
    
    For any $j=1,\dots,n-1$, the  irreducible representation of $G$ indexed by $j$ in item~\ref{item:dicyclic-2d-reps} or \ref{item:dihedral-2d-reps} above can be embedded in $\kk[V]_{\leq n+1}$ as the subspaces $\langle x^j,y^j\rangle_\kk$ and $\langle x^{j+1}y, xy^{j+1}\rangle_\kk$, of degrees $j$ and $j+2$. These are isomorphic as $G$-representations via $x^j \mapsto x^{j+1}y,\; y^j\mapsto -xy^{j+1}$ (in both cases $j$ odd and $j$ even), and the matrix
    \[
    \begin{pmatrix}x^j & x^{j+1}y \\ y^j & -xy^{j+1}\end{pmatrix}
    \]
    is nonsingular (because the characteristic of $\kk$ is different from $2$); it follows by Observation~\ref{obs:matrix-test-of-li} that the corresponding embeddings are linearly independent over $\kk(V)^G$.

    Since we have considered all the classes in $\Irr_\kk(G)$, this analysis shows that $\Dspan \leq n+1$. This is in fact an equality because the representation with $j=n-1$ described in item~\ref{item:dicyclic-2d-reps} or \ref{item:dihedral-2d-reps} above (depending on the parity of $n$) does not embed below degree $n-1$, and does not embed a second time below degree $n+1$. One can see this using the lattice methods of  Section~\ref{sec:preliminaries} (and \ref{sec:bfieldr} and \ref{sec:pf-of-HRD-main}), as follows. The action of $\sigma$ on $\kk[V]$ is diagonal, and in the representation in question, $\sigma$ has eigenvalues $\zeta^{n-1},\zeta^{n+1}$. To find two copies of the representation will thus require two different monomials on which $\sigma$ acts with the eigenvalue $\zeta^{n-1}$. Such monomials correspond to lattice points $\bfa =(a_1,a_2)$ in the first quadrant satisfying the condition $a_1 - a_2 = n-1\pmod{2n}$. The two nonnegative lattice points of lowest $L^1$ norm satisfying this condition are $(n-1,0)$ and $(n,1)$, corresponding to monomials $x^{n-1}$ and $x^ny$ of degrees $n-1$ and $n+1$.
\end{proof}

\begin{lemma}\label{lem:powerset-embedding}
    Let $\Omega$ be a meet semilattice (in the order-theoretic sense), with minimal element $\widehat 0$, and meet operation $\wedge$. Let $S$ be a finite subset of $\Omega$ such that $\wedge S=\widehat 0$, and that is inclusion-minimal with respect to this property. Then the map
    \begin{align*}
        2^S\setminus \{\varnothing\} &\rightarrow \Omega\\
        S\supseteq J&\mapsto \wedge J
    \end{align*}
    from the poset $2^S\setminus\{\varnothing\}$ of nonempty subsets of $S$ (with containment order) to $\Omega$, which by construction is order-reversing, is also injective.
\end{lemma}

\begin{proof}
    Suppose, for a contradiction, that $\wedge J = \wedge J'$ for two distinct, nonempty subsets $J,J'\subseteq S$. Then $U:=J\cup J'$ properly contains at least one of $J,J'$; without loss of generality, say $J$. Meanwhile, $\wedge U = (\wedge J)\wedge (\wedge J') = \wedge J$. But then 
    \[
    \widehat 0 = \wedge S=(\wedge U) \wedge (\wedge U^c) = (\wedge J) \wedge (\wedge U^c) = \wedge (J\cup U^c),
    \]
    contradicting the minimality of $S$ because $U\supsetneq J$ implies that $J\cup U^c$ is a proper subset of $S$.
\end{proof}

We are finally ready to prove Theorem~\ref{thm:Dspan-main}.

\begin{proof}[Proof of Theorem~\ref{thm:Dspan-main}]
    First, $\Dspan$ is insensitive to base field extension, by \cite[Lemma~2.2]{bbs-derksen}. Meanwhile, the properties of the representation $V$ mentioned in the theorem statement (being faithful, and not a sum of scalar and trivial subrepresentations) are also unaffected by base field extension. Of course the characteristic of $\kk$ is also unaffected. Thus, we can assume without loss of generality that $\kk$ is algebraically closed.

    We now consider five cases: 
    \begin{enumerate}
        \item $G$ is not cyclic or dicyclic;\label{case:not-cyclic-or-dicyclic}
        \item $G$ is cyclic and $V$ has a faithful subrepresentation of dimension $\leq 2$;\label{case:cyclic-small-faithful}
        \item $G$ is dicyclic and $V$ has a faithful subrepresentation of dimension $\leq 2$;\label{case:dicyclic-small-faithful}
        \item $G$ is cyclic but $V$ does not have a faithful subrepresentation of dimension $\leq 2$;\label{case:cyclic-no-small-faithful}
        \item $G$ is dicyclic but $V$ does not have a faithful subrepresentation of dimension $\leq 2$.\label{case:dicyclic-no-small-faithful}
    \end{enumerate}
    Case~\ref{case:not-cyclic-or-dicyclic} follows from the same result for $\topdeg$, which is essentially known. Case~\ref{case:cyclic-small-faithful} is proven via  Theorem~\ref{thm:HRD-main}. Case~\ref{case:dicyclic-small-faithful} is proven using Lemma~\ref{lem:dicyclic}. Cases~\ref{case:cyclic-no-small-faithful} and \ref{case:dicyclic-no-small-faithful} are then extracted by means of analyses that take advantage of Lemma~\ref{lem:powerset-embedding}.

    \textbf{Case~\ref{case:not-cyclic-or-dicyclic}: $G$ is not cyclic or dicyclic.}
    
    If $G$ is not cyclic or dicyclic, then the inequality $\Dspan(G,V)\leq |G|/2$ follows from the same bound on $\operatorname{topdeg}(G,V)$, which is essentially known, by the following considerations. Following established convention, let
    \[
    \topdeg(G):=\sup_V \topdeg(G,V)
    \]
    and
    \[
    \beta(G) := \sup_V \beta(G,V).
    \]
    By \cite[Theorem~1]{kohls2014top}, we have
    \[
    \topdeg(G) = \beta(G)-1.
    \]
    This and the following quoted results depend on the coprime hypothesis $\Char(\kk)\nmid |G|$. 
    
    By \cite[Theorem~1.1]{cziszter-domokos}, the classical Noether number $\beta(G)$ satisfies $\beta(G)<|G|/2$ (strict inequality) unless $G$ is isomorphic to $(\ZZ/3\ZZ)^2$, $(\ZZ/2\ZZ)^3$, the alternating group $A_4$,  the binary tetrahedral group $\tilde A_4$, or a group with a cyclic subgroup of index $2$. We have
    \[
    \beta((\ZZ/3\ZZ)^2) = 5,\; \beta((\ZZ/2\ZZ)^3) = 4
    \]
    by \cite[\S~2]{olson1969}, in view of the equality between $\beta(G)$ and the Davenport number of an abelian group \cite[Proposition~2.2]{schmid1991finite}, \cite[Section~1.4]{cziszter-domokos}. If $G=A_4$ then $\beta(G)=6$ by \cite[Theorem~3.4]{cziszter-domokos}, and if $G=\tilde A_4$ then $\beta(G)=12$ by \cite[Corollary~3.6]{cziszter-domokos}. If $G$ has a cyclic subgroup of order $2$, then
    \[
    \beta(G)=|G|/2+1
    \]
    unless $G$ is dicyclic, by \cite[Theorem~10.3]{cziszter2014noether}. Thus, combining all these cases, as long as $G$ is neither cyclic nor dicyclic we have
    \[
    \beta(G)\leq |G|/2 + 1, 
    \]
    whereupon
    \[
    \Dspan(G,V)\leq \operatorname{topdeg}(G)=\beta(G)-1\leq |G|/2
    \]
    by \cite[Theorem~1]{kohls2014top} and the fact that $\Dspan(G,V)\leq\operatorname{topdeg}(G,V)$.\footnote{If $G$ is dicylic, then the same line of reasoning gives $\Dspan(G,V)\leq \topdeg(G)=\beta(G)-1 = |G|/2 + 1$, which is almost the desired bound (but not quite).}

    \textbf{Case~\ref{case:cyclic-small-faithful}: $G=C_n$ is cyclic and $V$ contains a faithful subrepresentation of dimension $\leq 2$.}

    This is the case handled by Theorem~\ref{thm:HRD-main}, as follows.

    Since $\kk$ is algebraically closed and $\Char \kk$ does not divide $|G|$, $V$ splits into one-dimensional subrepresentations of $G=C_n$, and the hypothesis on $V$ in the statement of Theorem~\ref{thm:Dspan-main} implies that it contains at least two distinct nontrivial such one-dimensional subrepresentations. We can find a {\em faithful} two-dimensional subrepresentation $W$ that shares this property: either the faithful subrepresentation of dimension $\leq 2$ which exists by the hypothesis of the present case---call it $\tilde W$---is already exactly two-dimensional and already contains two distinct one-dimensional subrepresentations, and can thus be taken as $W$, or else $\tilde W$ is scalar or scalar $\oplus$ trivial, in which case it contains a faithful one-dimensional subrepresentation, to which a distinct one-dimensional representation somewhere else in $V$ (which exists by the hypothesis on $V$ in the statement of Theorem~\ref{thm:Dspan-main}) can be added to produce $W$.
    
    Now $W$ is a summand of $V$, and therefore a homomorphic image of $V$. Then it follows from \cite[Theorem~4.4]{bbs-derksen} that $\Dspan(G,V)$ is bounded above by $\Dspan(G,W)$. Therefore, the theorem follows from the special case that $V$ is itself a faithful two-dimensional representation whose one-dimensional summands are distinct and nontrivial.

    This case falls under the rubric of Lemma~\ref{lem:lattice-knows-all}, part~\ref{item:Dspan-from-lattice}: we can compute $\Dspan$ from the lattice $L(G,V)$. The index of $L(G,V)$ in $\ZZ^2$ is $n=|G|$, by Lemma~\ref{lem:index-is-|G|}. The fact that the summands of $V$ are distinct and nontrivial translates into $L(G,V)$ satisfying the hypothesis on $L$ in Theorem~\ref{thm:HRD-main}. So the conclusion of Theorem~\ref{thm:Dspan-main} follows from Theorem~\ref{thm:HRD-main} via Lemma~\ref{lem:lattice-knows-all} in this case.

    \textbf{Case~\ref{case:dicyclic-small-faithful}: $G=Dic_{4n}$ is dicyclic and $V$ contains a faithful subrepresentation of dimension $\leq 2$.}

    The hypothesized faithful subrepresentation of dimension $\leq 2$ must in fact be irreducible and of dimensison exactly $2$, as if it were either one-dimensional or reducible, it would factor through $G$'s abelianization, contradicting faithfulness. As in Case~\ref{case:cyclic-small-faithful}, by means of \cite[Theorem~4.4]{bbs-derksen} we reduce to the case that $V$ is itself a faithful irreducible 2-dimensional representation. Then, up to an automorphism of $G$, $V$ is isomorphic to the representation given by \eqref{eq:dicyclic}, so by Lemma~\ref{lem:dicyclic} we have
    \[
    \Dspan(G,V) = n+1 < 2n = |G|/2.
    \]

    \textbf{Case~\ref{case:cyclic-no-small-faithful}: $G=C_n$ is cyclic, but $V$ does not contain a faithful subrepresentation of dimension $\leq 2$.}

    Let
    \[
    V=V_1\oplus \dots \oplus V_r
    \]
    be a decomposition into irreducible subrepresentations of $G=C_n$; because $\kk$ is algebraically closed, they are all one-dimensional. As in Case~\ref{case:cyclic-small-faithful}, by means of \cite[Theorem~4.4]{bbs-derksen} we can reduce to the case that $V$ is {\em minimal faithful}, in the sense that omitting any $V_j$ leads to a representation that is not faithful. The hypothesis of the present case then translates into the statement that $r\geq 3$.

    For each $j=1,\dots,r$, let $G_j$ be the image of $G$ in $GL(V_j)$. Let $a_1,\dots,a_r$ be the orders of the kernels of $V_1,\dots,V_r$ respectively, so that $|G_j|=n/a_j$.

    The lattice, in the order-theoretic sense, of subgroups of $C_n$, is isomorphic to the (order-theoretic) lattice of positive integer divisors of $n$ with meet operation $\gcd$, by sending a subgroup to its order. Thus, faithfulness of $V$ implies that $\gcd(a_1,\dots,a_r)=1$, and minimal faithfulness of $V$ (in the sense defined above) implies that the set $\{a_1,\dots,a_r\}$ is inclusion-minimal with respect to this property among sets of divisors of $n$. Thus, by Lemma~\ref{lem:powerset-embedding} we have an order-reversing injective map 
    \[
    J\mapsto \gcd(\{a_j\}_{j\in J})
    \]
    from nonempty subsets $J\subset \{a_1,\dots,a_r\}$ to divisors of $n$. Because the (order-theoretic) lattice of divisors of $n$ is ranked, with rank function the number of prime factors including multiplicity, and meanwhile each singleton $\{a_j\}$ sits at the bottom of a chain of length $r-1$ in the poset of nonempty subsets of $\{a_1,\dots,a_r\}$, it follows that each $a_j$ has at least $r-1$ prime factors, counting multiplicity. 
    
    Furthermore, each $a_j$ has at least two {\em distinct} prime factors, as follows: if (for a contradiction) $a_j=p^\ell$ is a prime power for some $j$, then because $\gcd(a_1,\dots,a_r)=1$ there must be some other $a_{j'}$ not divisible by $p$. But then $\gcd(a_j,a_{j'})=1$, contradicting the containment-minimality of $\{a_1,\dots,a_r\}$ with respect to having $\gcd=1$, because $r\geq 3$. 

    Combining these conclusions, we find that
    \[
    a_j \geq 2^{r-2}\cdot 3
    \]
    for each $j$. Thus,
    \[
    \Dspan(G,V) \leq \sum_{j=1}^r \Dspan(G_j,V_j) \leq \sum_{j=1}^r \left(\frac{n}{a_j} - 1 \right)< \sum_{j=1}^r \frac{n}{a_j} \leq \frac{rn}{2^{r-2}\cdot3}\leq \frac{|G|}{2},
    \]
    where the first inequality is by \cite[Proposition~4.6]{bbs-derksen}, the second is by \cite[Theorem~4.7]{bbs-derksen}, and the final inequality is because ($n=|G|$ and) $r/(2^{r-2}\cdot 3) \leq 1/ 2$ for all $r\geq 3$.

    \textbf{Case~\ref{case:dicyclic-no-small-faithful}: $G=Dic_{4n}$ is dicyclic, but $V$ does not contain a faithful subrepresentation of dimension $\leq 2$.}

    The method is similar to Case~\ref{case:cyclic-no-small-faithful}. Let
    \[
    V=V_1\oplus \dots \oplus V_r
    \]
    be a decomposition into irreducibles; for $j=1,\dots,r$ let $G_j$ be the image of $G$ in $GL(V_j)$ and let $K_j$ be the kernel of the map $G\rightarrow G_j$. By means of \cite[Theorem~4.4]{bbs-derksen} we reduce to the case that $V$ is minimal faithful; thus $\bigcap_j K_j=\{\mathrm{id.}\}$, and the set $\{K_1,\dots,K_r\}$ is containment-minimal among subsets of the (order-theoretic) lattice of normal subgroups of $G$ with respect to this property. Because each $V_j$ has dimension at most $2$, the hypothesis of the present case implies that $r\geq 2$. 

    The normal subgroups of $G=Dic_{4n}$ are as follows. For each divisor $a$ of $2n$ there is exactly one normal subgroup of order $a$, generated by $\sigma^{2n/a}$. If $n$ is even, then there are two additional index-2 normal subgroups $\langle \sigma^2,\tau\rangle$ and $\langle \sigma^2, \sigma\tau\rangle$. In the odd $n$ case, because all the normal subgroups are contained in the cyclic group $\langle \sigma\rangle$ of order $2n$, the (order-theoretic) lattice of normal subgroups is isomorphic to the poset of divisors of $2n$, and we will be able to proceed as in Case~\ref{case:cyclic-no-small-faithful} (carried out below). For even $n$, the minimality condition on the $K_j$'s implies that the two subgroups not contained in $\langle \sigma \rangle$ also cannot occur as $K_j$'s, as follows. If $n$ is even, then a nontrivial factor of $2n$ cannot be relatively prime with $n$. Therefore, the intersection of a collection of subgroups of $\langle \sigma \rangle$ of order $2n$, with the subgroup $\langle \sigma^2 \rangle$ of order $n$, cannot be trivial, unless the intersection was already trivial before additionally intersecting with $\langle \sigma^2 \rangle$. Therefore, if either $\langle \sigma^2,\tau\rangle$ or $\langle \sigma^2, \sigma\tau\rangle$ appears among a collection of normal subgroups having trivial intersection, it can be dropped from the collection without affecting the intersection. So in this case as well, each $K_j$ is a subgroup of $\langle \sigma \rangle$, and we can again proceed as in Case~\ref{case:cyclic-no-small-faithful}. 

    Let $a_j := |K_j|$. As in Case~\ref{case:cyclic-no-small-faithful}, we have $\gcd(a_1,\dots,a_r)=1$, and $\{a_1,\dots,a_r\}$ is containment-minimal among sets of divisors of $2n$ with respect to this property. By Lemma~\ref{lem:powerset-embedding}, the map $J\mapsto\gcd(\{a_j\}_{j\in J})$ is an order-reversing injection from the nonempty subsets of $\{a_1,\dots,a_r\}$ into the poset of divisors of $2n$. Therefore, each $a_j$ has height at least $r-1$ in the poset of divisors of $2n$, so $a_j\geq 2^{r-1}$. Also, the containment-minimality of $\{a_1,\dots,a_r\}$ with respect to $\gcd(a_1,\dots,a_r)=1$ implies it is an antichain in the poset of divisors of $2n$; the fact that $r\geq 2$ then implies that the poset of divisors of $2n$ contains at least 2 noncomparable elements, so $2n$ and therefore $n$ is divisible by some prime other than $2$. It also implies that no $a_j$ equals $2n$ itself; thus we have $a_j\leq n$ for all $j$.

    We now compute $\Dspan(G_j,V_j)$ for the individual representations $V_j$. 
    \begin{enumerate}
        \item If $a_j$ is odd, then $K_j$ does not contain $\sigma^n$. Assuming this:
        \begin{enumerate}
            \item If $a_j\neq n$ then $G/K_j\cong G_j$ is a dicyclic group of order $4n/a_j$, and $V_j$ is a faithful representation. Then $\Dspan(G_j,V_j)$ can be computed from Lemma~\ref{lem:dicyclic}, because \eqref{eq:dicyclic} with $n/a_j$ in the place of $n$ is the only faithful representation of $G_j$ up to automorphisms of $G_j$. We have 
            \[
            \Dspan(G_j,V_j)=\frac{n}{a_j}+1
            \]
            in this case.
            \item If $a_j=n$ then $G/K_j\cong G_j$ is cyclic of order $4$, and $V_j$ is a faithful one-dimensional representation, so
            \[
            \Dspan(G_j,V_j)=3
            \]
            in this case.
        \end{enumerate}
        \item If $a_j$ is even, then $K_j$ contains the central involution $\sigma^n$. In this case it is impossible for $a_j$ to equal $n$ because this would imply $K_j=\langle \sigma^2\rangle$, but this is not the kernel of an irreducible representation of $G$ when $n$ is even because $G/\langle \sigma^2\rangle \cong C_2\times C_2$ has no faithful irreducible representations. Thus $G/K_j\cong G_j$ is a dihedral group of order $4n/a_j$, and $V_j$ is, up to automorphisms of $G_j$, its reflection representation (which is its unique faithful irreducible representation up to automorphisms of $G_j$). Thus, 
        \[
        \Dspan(G_j,V_j)=2n/a_j
        \]
        by Example~\ref{ex:sharp-dihedral} (with $2n/a_j$ in the place of $n$).
    \end{enumerate}
    Since $a_j\leq n$, we have $n/a_j + 1 \leq 2n/a_j$, and also, from $a_j\geq 2^{r-1}$ above we have $2n/a_j \leq n/2^{r-2}$. Also, Lemma~\ref{lem:powerset-embedding} implies that the height of the poset of divisors of $2n$, i.e., the number of prime factors of $2n$ counting multiplicity, is at least $r$ (because it lies above all $a_j$, of which there are at least two, and they are noncomparable and each have height at least $r-1$), and therefore $n$ has at least $r-1$ prime factors counting multiplicity. Meanwhile, we know from above that one of them is at least 3. It follows that $3 \leq n/2^{r-2}$. Therefore, in all cases, we have 
    \[
    \Dspan(G_j,V_j)\leq n/2^{r-2}
    \]
    for any $j=1,\dots,r$. Combining everything, we get
    \[
    \Dspan(G,V) \leq \sum_{j=1}^r \Dspan(G_j,V_j) \leq \frac{rn}{2^{r-2}} \leq 2n=\frac{|G|}{2}
    \]
    with the first inequality by \cite[Proposition~4.6]{bbs-derksen}, the second by what has been done above, and the third because ($4n=|G|$ and) $r/2^{r-2}\leq 2$ for $r\geq 2$.
\end{proof}

\section*{Acknowledgements}

The authors thank Aniah Serrano for research conversations in the early stages of this project, John Landa for assistance in setting up computations, Victor Reiner for suggesting the study of $\Dspan$, and Peter Symonds for calling our attention to \cite{kohls2014top}. We are grateful to the Institute for Pure and Applied Mathematics (IPAM) at UCLA, and the Petersen Automotive Museum, for their hospitality in hosting research visits in May 2024. Computations supporting this investigation were done in Python and Java. Drafts of Figures~\ref{fig:illustration-of-staircase-lemma} and \ref{fig:proof-idea-H-E-basis} were vibe-coded in TikZ using Claude 4.5 Opus (i.e., a back-and-forth with Claude 4.5 Opus was used to generate TikZ code for these figures from natural-language descriptions of the desired images); details in the code were then manually adjusted to arrive at the final images.

\bibliographystyle{alpha}
\bibliography{biblio}

\end{document}